\documentclass[10pt]{elsarticle}
\usepackage[margin= 1.85cm]{geometry}
\usepackage[linesnumbered]{algorithm2e}
\usepackage{algpseudocode}
\usepackage{amssymb, amsmath, amsthm, bm}
\usepackage{graphicx}
\usepackage[T1]{fontenc}
\usepackage[utf8]{inputenc}		
\usepackage{float}
\usepackage{epsfig}
\usepackage{multirow}
\usepackage{booktabs}
\usepackage[]{caption}
\usepackage{threeparttable}
\usepackage{subfigure}
\usepackage{ragged2e}
\usepackage{ifthen}
\usepackage[shortlabels]{enumitem}

\usepackage{esint}

\usepackage{etoolbox}
\patchcmd{\thebibliography}
  {\settowidth}
  {\setlength{\itemsep}{0pt plus 0.1pt}\settowidth}
  {}{}
\apptocmd{\thebibliography}
  {\small}
  {}{}

\usepackage[usenames,dvipsnames]{xcolor}
\usepackage[hidelinks,colorlinks]{hyperref}
\AtBeginDocument{
\hypersetup{
linkcolor = MidnightBlue, anchorcolor =red,
citecolor = ForestGreen,
filecolor = red,urlcolor = red,
            pdfauthor=author}}

\makeatletter
\newcommand\fake@math{}% just for safety
\def\fake@math#1\){[math]}
\makeatother

\usepackage{graphicx}
\usepackage{indentfirst}
\usepackage{multicol}
\usepackage{booktabs}
\usepackage{float} 
\usepackage{subfigure}

\newtheorem{theorem}{Theorem}[section]
\newtheorem{lemma}[theorem]{Lemma}
\newtheorem{proposition}[theorem]{Proposition}
\newtheorem{corollary}[theorem]{Corollary}

\numberwithin{equation}{section}

\theoremstyle{remark}
\newtheorem{remark}[theorem]{Remark}

\theoremstyle{definition}

\usepackage{enumitem}

\newcommand{\fnorm}[1]{\left\|#1\right\|_{{\rm F}}}

\newcommand{\bracket}[1]{\left(#1\right)}
\newcommand{\abs}[1]{\left|#1\right|}

\newcommand{\diag}[1]{\mathrm{diag}\left(#1\right)}
\newcommand{\w}[1]{\widetilde{ #1}}

\newcommand{\norm}[1]{\lVert #1 \rVert}

\newcommand{\mb}[1]{{\color{black} #1}}

\newcommand{\mc}[1]{\mathcal{#1}}

\def\q {\quad}

\def \l{\langle}
\def \r{\rangle}

\def \sss {\scriptscriptstyle}

\def\mm{ \left[
 \begin{matrix}}
\def\nn{\end{matrix} \right] }

\def \lad {\lambda}
\def \Lad{\Lambda}
\def \d{\delta}
\def \De{\Delta}
\def \ep {\varepsilon}
\def \vp {\varphi}

\def \na {\nabla}

\def \S{\mathbb{S}}
\def \R{\mathbb{R}}
\def \C{\mathbb{C}}

\def \on{{\rm O}(n)}
\def \sn {{\rm S}(n)}

\def \dd {\cdot}

\def \t {\times}

\def \si{\sigma}

\DeclareMathOperator{\hess}{Hess}
\DeclareMathOperator{\grad}{grad}
\DeclareMathOperator{\dist}{dist}

\begin{document}

\title{Vector-wise Joint Diagonalization of Almost Commuting Matrices}
% Hermitian Matrices with Unitary Similarity Transformations}

\begin{abstract}
This work aims to numerically construct exactly commuting matrices close to 
given almost commuting ones,
which is equivalent to the joint approximate diagonalization problem. We first prove that almost commuting matrices generically \mb{have} approximate common eigenvectors that are almost orthogonal to each other. Based on this key observation, we propose an efficient and robust vector-wise joint diagonalization (VJD) algorithm, which constructs the orthogonal similarity transform by sequentially finding these approximate common eigenvectors. In doing so, we consider sub-optimization problems over the unit sphere, for which we present a Riemannian approximate Newton method with a rigorous convergence analysis. We also discuss the numerical stability of the proposed VJD algorithm. Numerical experiments with applications in the independent component analysis are provided to reveal the relation with Huaxin Lin's theorem and to demonstrate that our method compares favorably with the state-of-the-art Jacobi-type joint diagonalization algorithm.
\end{abstract}

\begin{keyword}
Almost commuting matrices \sep Huaxin Lin's theorem \sep Joint diagonalizability \sep Riemannian approximate Newton method \sep Vector-wise algorithm
\MSC[2020] 15A18 \sep 15A27 \sep 15A45 \sep 15A60 \sep 58C15 \sep 65F15 \sep 65F35 
\end{keyword}

\date{}

\author[1]{Bowen Li\corref{mycorrespondingauthor}}\cortext[mycorrespondingauthor]{Corresponding authors}
\ead{bowen.li200@duke.edu}
\author[1,2]{Jianfeng Lu\corref{mycorrespondingauthor}}
\ead{jianfeng@math.duke.edu}
\author[3]{Ziang Yu\corref{mycorrespondingauthor}}
\ead{yuziang@sjtu.edu.cn}

\address[1]{Department of Mathematics, Duke University, United States}
\address[2]{ Department of Chemistry and Department of Physics, Duke University, United States}
\address[3]{School of Mathematical Sciences, Shanghai Jiao Tong University, China}

\maketitle

\section{Introduction}
It is a fundamental result in linear algebra that
a set of Hermitian matrices are commuting if and only if they are jointly diagonalizable by a unitary matrix. A closely related and long-standing question \cite{halmos1976some,rosenthal1969almost}, dating back to the 1950s, was that are almost commuting matrices near commuting pairs? It was answered affirmatively by 
Huaxin Lin \cite{Lin1997} in 1995, which gave the following celebrated theorem: for any $\epsilon>0$, there exists $\delta>0$, such that if $A,B\in\mathbb{C}^{n \times n}$ are Hermitian matrices satisfying $\norm{A}, \norm{B} \leq 1$ and $\norm{AB-BA}<\delta$, then there exist commuting Hermitian matrices $A^\prime,B^\prime$ such that $\norm{A^\prime-A},\norm{B^\prime-B}<\epsilon$. A simplified proof of Lin's result was provided soon after by 
Friis and R\o rdam \cite{friis1996almost}, and
an analogous result in the real symmetric case is recently given in \cite{Loring2016} by  Loring and S\o rensen. However, their arguments were abstract and did not show how to construct these matrices $A', B'$, which would be very useful from the 
application point of view \cite{Loring2016}. In 2008, Hastings \cite{Hastings2009} suggested a constructive approach to finding the nearby commuting matrices $A'$ and $B'$ and discussed the dependence of $\d$ on $\epsilon$ (see also \cite{Herrera2020}
for the extended discussions about Hastings' method).
Later, Kachkovskiy and Safarov, in their seminal work
\cite{kachkovskiy2016distance}, revisited this problem in the framework of operator algebra theory and gave the optimal dependence of $\d$ on $\epsilon$ by a  different argument from the one in \cite{Hastings2009}. We state their result below for future reference. The interested readers are referred to
\cite{Herrera2020,kachkovskiy2016distance} for a more complete historical review of Lin's theorem and its development. \mb{A list of notations used throughout this work is
given at the end of this section.}

\begin{theorem}[\cite{kachkovskiy2016distance}] \label{linthm}
For any two self-adjoint operators $A,B$ on a Hilbert space $H$, there exists a pair of commuting self-adjoint operators $A',B'$ such that 
\begin{align} \label{eq:linest}
    \norm{A - A'} + \norm{B - B'} \le C \norm{[A,B]}^{1/2}\,,
\end{align}
where \mb{$[A,B] = AB - BA$ is the commutator} and 
the constant $C$ is independent of $A,B$ and the dimension of $H$.
\end{theorem}

The theory of almost commuting matrices and Lin's theorem find numerous applications in many fields, e.g., quantum mechanics \cite{hastings2008topology,Hastings2010,hastings2011topological,ogata2013approximating}, computer graphics \cite{Bronstein2013,kovnatsky2013coupled}, and signal processing \cite{Cardoso1993,Feng2003,Rutledge2013}. In particular, Hastings used Lin's theorem to claim the existence of localized Wannier functions for non-interacting fermions \cite{hastings2008topology}, which inspired a series of subsequent works \cite{Hastings2010,hastings2011topological}. For the application in computer science,
there is a famous algorithm called \emph{Joint Approximate Diagonalization} (JADE) developed by Cardoso and Souloumiac for blind source separation \cite{Cardoso1993,Cardoso1996}. Such an
algorithm has also been used in computational quantum chemistry to generate Wannier functions \cite{gygi2003computation}. The machine learning-related applications are recently discussed in \cite{glashoff2013matrix,kovnatsky2013coupled}. 
While numerous attempts have been made to constructively prove Lin's theorem and explore its applications in various areas, a numerically feasible algorithm for finding $A',B'$ in Theorem \ref{linthm} remains elusive.

In this work, our focus is on designing a fast algorithm for Theorem \ref{linthm} in the almost commuting regime. Given matrices $A,B \in \sn$ with $\norm{[A,B]} \ll 1$, we consider the following equivalent optimization problem:

\begin{equation}\label{eq:problem}
\begin{split} 
    & \min_{D_1,\, D_2 \in \diag{n},\, U \in O(n)} \mathcal{J}(D_1, D_2, U) := \norm{U^T AU - D_1}^2 + \norm{U^T B U - D_2}^2.
    % , \\ 
    % & \text{s.t. } U^*U=I,
\end{split}  \tag{OP}
\end{equation}
% \mb{See the notation list below for notations used here.} 
In view of practical applications
\cite{eynard2012multimodal,glashoff2013matrix,gygi2003computation}, we will focus on real matrices, but all the results and discussions can be adapted to the complex case easily.  
It may be clear but worth emphasizing that any approximate solution to \eqref{eq:problem} gives a potential pair of commuting matrices $(A',B') = (U D_1 U^T, U D_2 U^T)$ satisfying the estimate \eqref{eq:linest}. The problem \eqref{eq:problem} is generally very difficult to solve by the existing manifold optimization techniques, since the operator norm involved is closely related to the Finsler geometry \cite{bhatia2007spectral}, which is a rather uncommon structure in the optimization community. Even if we replace the operator norm with the Frobenius one, \eqref{eq:problem} is still a challenging problem, due to the non-convexity of the objective function and the Stiefel manifold constraint. One old but popular method is a Jacobi-type algorithm represented in \cite{B-G1993} with ideas dating back to \cite{eberlein1970solution,goldstine1959procedure}. This algorithm minimizes the off-diagonal objective function: for Hermitian matrices $A$ and $B$, 
\begin{equation*} 
J(A,B,U) = {\rm off}(U^* A U) + {\rm off}(U^* B U)\,,    
\end{equation*}
by a sequence of unitary transformations constructed from Givens rotation matrices:
\begin{equation}\notag
   R(i,j,c,s) := I + (c-1)e_ie_i^T - \bar{s}e_ie_j^T + se_je_i^T + (\bar{c}-1)e_je_j^T\,,
\end{equation}
\mb{where $U$ is a unitary matrix with the complex adjoint $U^*$, ${\rm off}(A) := \sum_{\sss i \neq j}|a_{ij}|^2$ measures the off-diagonal part of a square matrix $A$, and parameters $c,s \in \C$ are the Jacobi angles satisfying $\abs{c}^{\sss 2} + \abs{s}^{\sss 2} = 1$.} The optimal $c$ and $s$ are derived in \cite[Theorem 1.1]{Cardoso1996} in a closed-form expression. For the readers' convenience, we briefly summarize the main procedures of the Jacobi algorithm \cite{B-G1993} in the following pseudocode.

\vspace{2mm}
\RestyleAlgo{ruled}
\begin{algorithm}[!htbp]
\caption{Jacobi algorithm for simultaneous diagonalization \cite{B-G1993,Cardoso1996}}\label{alg:jacobi}
\KwIn{$\epsilon > 0$; Hermitian commuting matrices $A, B$}
\KwOut{Unitary matrix $U$ such that ${\rm off}(U^*AU) + {\rm off}(U^*B U) \le \epsilon (\norm{A}_{\rm F} + \norm{B}_{\rm F})$}

$U = I$\;
\While{${\rm off}(A) + {\rm off}(B) > \epsilon(\fnorm{A}+\fnorm{B})$}{
    \For{$i = 1,\cdots,n$ and $j = i+1,\cdots,n$}{
        set parameters $c$ and $s$ as described in \cite{Cardoso1996}
        
        $R = R(i,j,c,s)$
        
        $U = UR$; $A = R^*AR$; $B = R^*BR$
    }
}
\end{algorithm}
\vspace{2mm}

Algorithm \ref{alg:jacobi} is known to have the quadratic asymptotic convergence rate \cite{B-G1993} and can be efficiently implemented in the parallel framework as discussed in \cite{berry1986multiprocessor,eberlein1987one}. Another advantage of the Jacobi algorithm is its numerical stability against rounding errors. It is worth noting that Theis \cite{theis2006towards} extended the above Jacobi algorithm for the joint block diagonalization, but the output block decomposition may not be the finest one.  For the finest block diagonalization, a class of algorithms based on the theory of matrix $*$-algebra 
has been proposed in \cite{de2011numerical} and developed in \cite{maehara2010numerical,maehara2011algorithm,murota2010numerical}. Other joint diagonalization techniques in various problem settings include a gradient-based algorithm for approximately diagonalizing almost commuting Laplacians \cite{Bronstein2013}, subspace fitting methods with non-orthogonal transformations \cite{Van2001, Yeredor2000}, and the fast approximate joint diagonalization free of trivial solutions \cite{Li2007}, just to name a few. \mb{In particular, Riemannian optimization algorithms have been recently explored for the joint diagonalization problem with or without orthogonality constraints. Theis et al. \cite{theis2009soft} applied the Riemannian trust-region method to the joint diagonalization on the Stiefel manifold. Within the same framework, an efficient Riemannian Newton-type method was investigated in \cite{sato2017riemannian} by Sato. 
More recently, Bouchard et al. \cite{bouchard2020approximate} proposed a unified Riemannian optimization framework for the approximate joint diagonalization on the general linear group, which generalizes previous algorithms \cite{absil2006joint,afsari2004some}.}

\mb{The joint matrix diagonalization problem is also closely related to the canonical polyadic decomposition (CPD). In 
\cite{de2006link}, De Lathauwer derived a sufficient condition for the uniqueness of CPD and showed that the decomposition can be obtained by a simultaneous diagonalization by congruence. S\o rensen \cite{sorensen2012canonical} proved that for a given 
tensor, the optimal low-rank CPD approximation exists with one of the matrix factors being column-wise orthonormal. In particular, numerical algorithms based on joint congruence diagonalization and alternating least squares are also explored in \cite{sorensen2012canonical} for this constrained CPD. We also mention that in a recent work \cite{evert2022guarantees} of Evert and De Lathauwer, they discussed the optimal CPD of a noisy low-rank tensor and connected it to a joint generalized
eigenvalue problem.}

In this work, we propose a practically efficient vector-wise algorithm for finding the approximate solution to the optimization problem \eqref{eq:problem} in the case where $\norm{[A, B]} \ll 1$, and present the associated error analysis. Numerical experiments demonstrate that our algorithm has comparable or even better performance than that of the state-of-the-art Jacobi-type method while significantly reducing computational time. 
Moreover, we numerically compute the error between the given almost commuting matrices $(A, B)$ and the commuting ones $(A', B')$ constructed by our algorithm. It turns out that in the almost commuting regime, $\norm{A - A'} + \norm{B - B'}$ is of the order $\norm{[A,B]}$, and thus the constructed $(A',B')$ satisfies the upper bound \eqref{eq:linest} in Lin's theorem. Some applications in independent component analysis are also provided to justify the effectiveness of our method.

We now summarize the main idea and contributions of this work in detail. Our method is motivated by the simple fact that the commutative matrices $A,B$ can be simultaneously diagonalized by an orthogonal matrix whose columns consist of common orthonormal eigenvectors of $A,B$, which can be characterized by the global minimizers of the following optimization problem:
\begin{equation}\label{eq:sub-problem}
\begin{split} 
    & \mb{\min_{\Lad = (\lad,\mu) \in \R^2\,, v\in \S^{n-1}}  \mathcal{L}(\Lad,v)}: = \norm{A v - \lad v}^2 + \norm{B v - \mu v}^2\,.
\end{split} \tag{${\rm OP}_1$}
\end{equation} 
Note that $\mc{L}(\Lad,v)$ admits $n$ minimizers $(\Lad_j, v_j)$ with $\mc{L}(\Lad_j, v_j) = 0$ and $v_j$ 
being mutually orthonormal. This observation leads to an ideal approach for diagonalizing commuting matrices $(A,B)$. Suppose that we have successfully solved the first $i$ minimizers $\{(\Lad_j, v_j)\}_{j = 1}^i$ of $\mc{L}(\Lad,v)$. To determine the $i+1$-th minimizer, we solve \eqref{eq:sub-problem} but with an additional orthogonal constraint:
\begin{align} \label{eq:subsub_pro}
    \min \left\{\mathcal{L}(\Lad, v)\,;\  \Lad \in \R^2\,,\  v \in \S^{n-1}\,,\  \l v, v_j \r = 0\,, 1 \le j \le i \right\}\,.
\end{align}
However, this elegant approach may encounter practical challenges due to numerical noise and errors. Since we can not solve the optimization problem \eqref{eq:subsub_pro} exactly at each step, the numerical errors will accumulate so that the constraints $\l v, v_j \r = 0$ may exclude the minimizers of 
the cost functional $\mc{L}$, which leads to the failure of the algorithm for finding the complete orthonormal set of common eigenvectors. Moreover, the input commuting matrices are generally subject to numerical noises or rounding errors. As proved in Lemma 
\ref{lem:opendensenocommon}, any such small perturbation could disrupt all the common eigenvectors. Consequently, the functional $\mathcal{L}$ might not exhibit sufficient local minimizers, in which case the algorithm will also fail.  

Our first contribution is Theorem  \ref{thm:generic_almost}, which proves that 
in the almost commuting regime, $\mc{L}(\Lad,v)$ generically has $n$ local minimizers $\{(\Lad_j, v_j)\}_{i = 1}^n$ with $v_j$ being almost orthogonal to each other (which may be viewed as \emph{approximate common eigenvectors}). For this, we provide a stability result in Theorem \ref{them:stab} by using perturbation techniques. Then, based on the theory of Gaussian orthogonal ensemble (GOE) \cite{anderson2010introduction,tao2012topics},  we show that the common eigenspaces of 
commuting matrices are generically of one dimension from both topological and probabilistic points of view (cf.\,Lemma \ref{lem:open_dense_common} and Corollary \ref{lem:high_prob_common}), similarly to the well-known genericity of simple eigenvalues of self-adjoint operators \cite{tao2017random,teytel1999rare,uhlenbeck1976generic}. 
Combining these results with the estimates of the spectral gaps for random matrices in \cite{nguyen2017random}, we can immediately conclude the desired claim in Theorem  \ref{thm:generic_almost}.

In view of this generic property of the local minimizers of $\mc{L}(\Lad,v)$, we consider a relaxed version of the 
aforementioned approach for finding the columns of the orthogonal matrix $U$ and the diagonal entries of matrices $\Lad_i$ sequentially, called the \emph{vector-wise joint diagonalization} (VJD); see Algorithm \ref{alg:VJD}. The core idea behind VJD is to relax the exact orthogonality constraint $\l v, v_j \r = 0$ in \eqref{eq:subsub_pro} to $|\l v, v_j \r| \le \ep$. This relaxation accommodates possible numerical errors and non-commutativity of input matrices. However, enforcing constraints of the form $|\l v, v_j \r| \le \ep$, $1 \le j \le i$, can be numerically challenging. Therefore, in practice, we adopt the equivalent constraint: 
\begin{equation} \label{eq:const}
    {\rm dist}(v, S_i)  \le \sqrt{\ep}\,, 
\end{equation}
where $S_i = {\rm span}\{v_1,\cdots,v_i\}^\perp \cap \S^{n-1}$ and $\ep$ is the error tolerance parameter. Notably, this form \eqref{eq:const} allows for the straightforward computation of the associated projection, as shown in Lemma \ref{lem:projection}. The VJD algorithm proceeds as follows; see Section \ref{sec:algorithm} for more details. We first solve \eqref{eq:sub-problem} to find $(\lad_1, \mu_1, v_1)$. Then, for $j \ge 2$,
we consider the relaxed problem of \eqref{eq:subsub_pro}: 
\begin{equation}\label{eq:problem with inequality constraint}
 \min \left\{\mathcal{L}(\Lad, v)\,;\  \Lad = (\lad,\mu) \in \R^2\,,\  v \in \S^{n-1}\,,\  {\rm dist}(v, S_{j-1}) \le \sqrt{\ep} \right\}\,, \tag{${\rm OP}_2$}
\end{equation}
which gives $(\lad_j,\mu_j,v_j)$. This iterative process culminates in the construction of an almost orthogonal matrix $V = [v_1,\cdots,v_n]$ and two diagonal matrices: $D_1 = {\rm diag}(\lad_1,\ldots,\lad_n)$ and $D_2 = {\rm diag}(\mu_1,\ldots,\mu_n)$. The algorithm ends up with a classical nearest orthogonal matrix problem (i.e., orthogonal Procrustes problem \cite{schonemann1966generalized}): 
\begin{align} \label{eq:orthogonalize}
    \min \{\norm{U - V}\,;\ U \in \on\} \,,
\end{align}
where $\norm{\dd}$ is either the spectral norm $\norm{\dd}_{{\rm op}}$ or the Frobenius norm $\norm{\dd}_{{\rm F}}$. In both cases, the problem \eqref{eq:orthogonalize} has an analytic solution $U = U_1U_2^T$, with $V=U_1\Sigma U_2^T$ being the singular value decomposition (SVD) of $V$. \mb{For solving sub-optimization problems \eqref{eq:sub-problem} and \eqref{eq:problem with inequality constraint}, we derive a Riemannian approximate Newton-type method with the approximate Hessian matrix being structurally simple and positive semidefinite; see Algorithm \ref{alg:for ineq}.} Our method exploits the second-order information of the cost function and enjoys stability with low computational cost. We would like to remark that there might be some other algorithms that are also potentially suitable for our optimization problems, for example, the trust-region method \cite{Absil2007}, the column-wise block coordinate descent method \cite{Gao2018,gao2019parallelizable}, the Cayley transform-based optimization \cite{Wen2013}, and the consensus-based optimization method \cite{fornasier2021consensus,ha2022stochastic,kim2020stochastic}, but the detailed discussion about these approaches is beyond the scope of this work.

Our third contribution lies in establishing the global convergence of the proposed Riemannian approximate Newton method for \eqref{eq:sub-problem} and demonstrating the numerical stability of our VJD algorithm.
\mb{By leveraging techniques from \cite{bortoloti2022efficient,li2019convergence,yang2007globally} for Newton-type methods on manifolds, we prove in Theorem \ref{thm:glob_conv} that Algorithm  \ref{alg:for ineq} for \eqref{eq:sub-problem} converges linearly with the convergence rate $O(\norm{[A,B]}^{1/2}) \ll 1$
and quadratically if $[A,B] = 0$.}  However, it is important to note that our optimization process, owing to its sequential nature, does not exactly correspond to solving \eqref{eq:problem} when $A$ and $B$ are non-commuting matrices. Nevertheless, based on empirical observations from numerical experiments, our method provides a good approximation to the original problem. To further substantiate this numerical advantage, we theoretically establish the error-tolerant property of our VJD algorithm in Proposition \ref{prop:numeri_sta}.
\mb{We remark that although our discussions primarily focus on a pair of almost commuting matrices, our algorithm and analysis can be directly extended to a family of matrices that are almost commuting pairwisely.}

This work is organized as follows. Section \ref{sec:stability analysis} delves into the stability and generic behaviors of local minimizers of $\mc{L}(\Lad,v)$, serving as motivation for the VJD algorithm outlined in Section \ref{sec:algorithm}, which relies on an efficient approximate Newton method for sub-optimization problems. Additionally, Section \ref{sec:error analysis} is dedicated to the examination of the convergence properties and numerical stability of the presented algorithms. In Section \ref{sec:experiments},  various numerical examples are presented to justify the efficiency of our method.

\medskip

\noindent \textbf{Notations.}

\begin{enumerate}[\textbullet]
    \item  
   Let $\sn$ be the set of symmetric $n \t n$ matrices, and $\on$ be the orthogonal group in dimension $n$. We denote by ${\rm diag}(n)$ the space of $n \t n$ diagonal matrices, and by ${\rm diag}_{\le}(n)$ the closed subset of ${\rm diag}(n)$ with increasing diagonal entries $\lad_1 \le \cdots \le \lad_n$. By abuse of notation, we also write $\diag{\lad_1,\ldots,\lad_n}$ for the diagonal matrix with diagonal entries  $\{\lad_i\}_{i = 1}^n$.
    \mb{We denote $A \succeq 0$ (resp., $A \succ 0$) for $A \in \sn$ if $A$ is positive semidefinite (resp., positive definite).}
    \item For a matrix $M \in \R^{m \t n}$, $M_{ij}$ denotes its $(i,j)$ element. Moreover, we denote by $\norm{\dd}_{{\rm F}}$ the Frobenius norm, and by $\norm{\dd}_{{\rm op}}$ (or, simply, $\norm{\dd}$) the operator norm.
\item Let $\l \dd, \dd\r$ be the Euclidean inner product of $\R^n$ and $\norm{\dd}$ be the associated norm. \mb{The standard basis of $\R^n$ is given by $\{e_i\}_{i = 1}^n$, where $e_i$ denotes the vector with $1$ in the $i$-th coordinate and $0$ elsewhere. We will not distinguish the row vector $(x_1,\ldots,x_n)$ and the column one $(x_1,\ldots,x_n)^T$ unless otherwise specified.}
We define the $(n-1)$-dimensional unit sphere by $
    \S^{n-1} := \{ x = (x_1,\ldots,x_n) \in \R^n\,;\ \norm{x} = 1\}.
$
\item We denote by $\si(X)$ the spectrum of a matrix $X \in \sn$ and by $\d(X)$ its spectral gap:
$$
    \d(X) = \min\{|\lad_i - \lad_j|\,;\ \lad_i \neq \lad_j\,,\  \lad_i,\lad_j \in \si(X)\}\,.
$$
\mb{We say that an eigenvalue $\lad \in \si(X)$ of $X \in \sn$ is simple (resp., degenerate) if its multiplicity is equal to (resp., greater than) one.}  
\item \mb{We recall the commutator $[A, B] = AB - BA$ of two matrices $A, B \in \R^{n \times n}$ and define the set of commuting matrices by 
$${\rm F}_{\rm com} := \{(A,B) \in \sn \t \sn \,; \ [A,B] = 0\}\,.$$ Then, we say that $\Lad := (\lad, \mu) \in \si(A) \times \si(B)$ is an \emph{eigenvalue
pair} of $(A,B) \in {\rm F}_{\rm com}$ if the associated common eigenspace is non-empty:
\begin{equation}\label{defcomm}
E_{A,B}^{\Lad} := \{v \in \R^n\,;\ A v = \lad v,\, B v = \mu v\} \neq \emptyset\,. 
\end{equation}
With slight abuse of notation, we define the joint spectrum $\si(A,B) \subset \R^2$ of $(A, B) \in {\rm F}_{\rm com}$ by the set of all its eigenvalue pairs $\Lad_j$ and the eigenvalue gap by 
\begin{equation} \label{eqgapp}
\d(A,B) := \min\{\norm{\Lad_i - \Lad_j}\,;\ \Lad_i \neq \Lad_j\,,\  \Lad_i, \Lad_j \in \si(A,B)\}\,.
\end{equation}
We also denote by ${\rm conv}(\si(A,B)) \subset \R^2$ the convex hull of $\si(A,B)$.
% \begin{equation*}
%    {\rm conv}(\si(A,B)) := \Big\{\sum_{i = 1}^n p_i \Lad_i\,;\ p_i \in [0,1],\ \sum_i p_i = 1, \ \Lad_i \in \si(A,B) \Big\}\,. 
% \end{equation*}
The multiplicity of $\Lad \in \si(A,B)$ is defined by the dimension of the subspace $E_{A,B}^{\Lad}$, and 
$\Lad \in \si(A,B)$ is said to be simple (resp., degenerate) if $\dim(E_{A,B}^{\Lad})$ is equal to (resp., greater than) one.} 
\end{enumerate}

\section{Analysis of local minimizers of \(\eqref{eq:sub-problem}\)}
\label{sec:stability analysis}

This section is devoted to the analysis of the local minimizers of the functional $\mc{L}(\Lad,v)$ in \eqref{eq:sub-problem}. \mb{Note $\mc{L}(\Lad,v) = \mc{L}(\Lad,-v)$ and that if $(\Lad,v)$ is a local minimum of $\mc{L}$, so is $(\Lad, -v)$. It is convenient to consider a point $(\Lad,v) \in \R^2 \t \S^{n-1}$ as an equivalence class $\{(\Lad,v), (\Lad,-v)\}$ in what follows. We shall first derive the optimality conditions of \eqref{eq:sub-problem} and characterize the minimizers of $\mc{L}$ in the commuting case.} Then, we
establish some stability estimates on the minimizers by perturbation arguments. We also show that with high probability $1-o(1)$, 
the functional $\mc{L}(\Lad,v)$ for almost commuting random matrices has $n$ local minima $\{(\Lad_i,v_i)\}_{i = 1}^n$ on the manifold $\R^2 \t 
\mathbb{S}^{n-1}$ with minimizing vectors $v_i \in \mathbb{S}^{n-1}$ almost orthogonal to each other. 
These findings shed light on designing the column-wise algorithm for the optimization problem \eqref{eq:problem}, which we will discuss in detail in the next section. 

Recall that in this work, we consider the almost commuting symmetric matrices $(A,B)$ with $\norm{[A,B]} \ll 1$. By Theorem \ref{linthm}, without loss of generality, we assume that the matrices $(A,B)$ can be written as 
\begin{align} \label{eq:pertburab}
    A = A_* + \Delta A\,, \q B = B_* + \Delta B\,,
\end{align}
for some commuting matrices $(A_*, B_*) \in {\rm F}_{\rm com}$ with
\begin{equation} \label{auxeqlin}
\norm{\De A} + \norm{\De B} \le C \norm{[A,B]}^{1/2} \ll 1\,.    
\end{equation}
Note that in the case of $[A,B] = 0$ (i.e., $\De A = \De B = 0$), we can always find $n$ local (also global) minima of $\mc{L}(\Lad,v)$ (with minimum value zero) given by the eigenvalue pairs of $(A,B)$ and the associated common orthonormal eigenvectors. \mb{Indeed, the converse is also true; see Proposition \ref{prob:characmini} below. 

We start with the optimality conditions for \eqref{eq:sub-problem}. For ease of exposition, we recall some basic concepts on Riemannian optimization on a smooth manifold $\mc{M}$ embedded in an Euclidean space $\R^n$ \cite{Absil2009,boumal2022intromanifolds}. We denote by $T_x \mc{M}$ the tangent space at $x \in \mc{M}$ and by $\mc{P}_{x}$ the orthogonal projection from $\R^n$ to $T_x \mc{M}$. Let $f$ be a smooth function on $\mc{M}$, which admits a smooth extension $\bar{f}$ to a neighborhood of $\mc{M}$ in $\R^n$. In what follows,
we will denote by $\na$ and $\na^2$ the standard Euclidean gradient and Hessian, respectively. Then, the Riemannian gradient of $f$ is given by 
\begin{align} \label{eq:riegrad0}
    \grad f(x) = \mc{P}_x (\na \bar{f}(x))\,;
\end{align}
see \cite[Proposition 3.53]{boumal2022intromanifolds}. Letting $\bar{G}$ be any smooth extension of $\grad f$, the Riemannian Hessian of $f$, as a linear map on $T_x \mc{M}$, can be computed by \cite[Corollary 5.14]{boumal2022intromanifolds}
\begin{align} \label{eq:riehess0}
    \hess f(x) = \mc{P}_x(\na \bar{G}(x))\,.
\end{align}
In the special case of $\mc{M} = \S^{n-1}$,} we have 
the tangent space $T_v \S^{n - 1} = \{p \in \R^{n}\,;\ \l p, v\r = 0\}$ at $v \in \S^{n-1}$ with the associated projection $\mc{P}_v$ given by $I_n - v v^T$. Thanks to \eqref{eq:riegrad0} and \eqref{eq:riehess0} above, the Riemannian gradient and Hessian of $f$ on $\S^{n-1}$ can be explicitly computed as 
\begin{align} \label{eq:riegrad}
    \grad f(v) = (I_n - v v^T) \na \bar{f}(v)\,, \q  \hess f(v) = (I_n - v v^T) (\na^2 \bar{f}(v) - \l \na \bar{f}(v), v\r I_n)\,.
\end{align}
We also recall the notion of retraction on $\S^{n-1}$ \cite[Definition 3.41]{boumal2022intromanifolds}. In this work, we limit our discussion to the following one for its simplicity (another popular choice is the exponential map):
\begin{align} \label{def:retraction}
    R_v(p) = \frac{v + p}{\norm{v + p}}\,,\q v \in \S^{n-1}\,,\ p \in T_v \S^{n-1}\,.
\end{align}
% Another popular one is the exponential map:
% \begin{equation*}\label{eqexpmap}
% \exp_v(p) = \cos(\norm{p}) v + \sin(\norm{p}) \frac{p}{\norm{p}}\,.
% \end{equation*} 
It is worth noting that $R_v$ in \eqref{def:retraction} is a second-order retraction on $\S^{n-1}$, i.e., for any $v \in \S^{n-1}$ and $p \in T_v \S^{n-1}$, the curve $c(t) = R_v(tp)$ satisfies $c(0) = 0$, $c'(0) = p$, and $c''(0) = 0$; see \cite[Definition 5.41]{boumal2022intromanifolds}.

We now compute the Riemannian gradient of the functional $\mc{L}(\Lad,v)$ on the manifold $\R^2 \t \S^{n-1}$ by  \eqref{eq:riegrad0}. Note that $\mc{L}$ is well-defined on $\R^{2 + n}$. For $\Lad = (\lad,\mu) \in \R^2$ and $v \in \S^{n-1}$, we have
\begin{align} \label{eq:gradl}
    \grad\mc{L}(\Lad,v) =  \mm 
   2 \lad - 2 \l v, Av \r \\
   2 \mu -  2 \l v, Bv\r \\
    \grad_v \mc{L}(\Lad,v) 
    \nn \q \text{with}\q   \grad_v \mc{L}(\Lad,v) = 2 \left(I_n - vv^T\right) \left(
    (A - \lad)^2 + (B - \mu )^2\right) v\,,
\end{align}
where $\grad_v \mc{L}$ denotes the Riemannian gradient of $\mc{L}(\Lad,v)$ in $v \in \S^{n-1}$. \mb{Similarly, by \eqref{eq:riehess0}, we calculate the Riemannian Hessian of $\mc{L}$ as follows:}
\begin{align} \label{eq:hessl}
    \hess\mc{L}(\Lad,v) & = 
    \mm 
    2 I_2 & W^T \\
   W & \hess_v\mc{L}(\Lad,v)
    \nn,
    % \\ 
    % & = \mm 
    % M & - 2(I_n - v v^T) A v &  - 2(I_n - v v^T) B v \\
    % - 2  ((I_n - v v^T) A v)^T & 1 & 0 \\
    % - 2  ((I_n - v v^T) B v)^T & 0 & 1
    % \nn
\end{align}
where the matrices $W \in \R^{n\t 2}$ and $\hess_v\mc{L}(\Lad,v) \in \sn$ are given by 
\begin{align} \label{eq:wwhess}
    W = - 4 \mm 
    \left(I_n - v v^T \right)A v & \left(I_n - v v^T \right) Bv  \nn,
\end{align}
and 
\begin{align} \label{eq:rieman_Lv}
    \hess_v\mc{L}(\Lad,v) =  (I_n - v v^T) \left( \na^2_v \mc{L}(\Lad,v) - \l v, \na^2_v \mc{L}(\Lad,v) v\r I_n \right).
\end{align}
Here, $\na^2_v \mc{L}(\Lad,v)$ is the Euclidean Hessian of $\mc{L}(\Lad,v)$ in $v$:
\begin{align}\label{eq:eucl_hess}
\na^2_v \mc{L}(\Lad,v) = 2 (A - \lad)^2 + 2 (B - \mu )^2\,.
\end{align}

\mb{We define the stationary points of \eqref{eq:sub-problem} by $\{(\Lad, v) \in \R^2 \times \S^{n-1}\,;\ \grad \mc{L}(\Lad,v) = 0\}$. Then, if $(\Lad, v)$ is a local minimizer of \eqref{eq:sub-problem}, then it is a stationary point and there holds $\hess \mc{L}(\Lad,v) \succeq 0$ on the tangent space $T_{(\Lad,v)} (\R^2 \t \S^{n-1}) = \R^2 \t T_v \S^{n-1}$. Conversely, if $(\Lad,v)$ is a stationary point such that $\hess \mc{L}(\Lad,v)
\succ 0$ on $\R^2 \t T_v \S^{n-1}$, then it is an isolated local minimizer of \eqref{eq:sub-problem}. We next give a characterization for the stationary points of \eqref{eq:sub-problem} for commuting matrices. We denote by $\Delta_n$ the $n$-simplex: 
\begin{equation*}
    \Delta_n = \Big\{(p_i)_i \in \R^n\,;\ p_i \in [0,1],\ \sum_{i = 1}^n p_i = 1 \Big\}\,.
\end{equation*}

\begin{proposition} \label{lem:statpoint}
Given $(A,B) \in {\rm F}_{\rm com}$, let $\{\Lad_i\}_{i = 1}^n$ be its eigenvalue pairs (counting multiplicity) and $\{v_i\}_{i = 1}^n \subset \S^{n-1}$ be the associated common orthonormal eigenvectors. 
We define the subset of $\Delta_n$:
\begin{equation} \label{eqsimstat}
    \Delta_{\rm sta}: = \Big\{(p_i)_i \in \Delta_n\,;\ \Lad = \sum_{i = 1}^n p_i \Lad_i\,,\ \norm{\Lad - \Lad_i}\ \text{is constant for}\ \{i\,; \  p_i \neq 0\}\Big\}\,.
\end{equation}
Then, the stationary points of \eqref{eq:sub-problem} are given by
\begin{equation} \label{eq:repsta}
   \Lad = \sum_{i = 1}^n p_i \Lad_i \in {\rm conv}(\si(A,B))\,,\  v = \sum_{i = 1}^n c_i v_i \in \S^{n-1}\,,
\end{equation}
with $(p_i)_i \in \Delta_{\rm sta}$ and $c_i = \pm \sqrt{p_i}$. 
\end{proposition}

\begin{proof}
It is clear from \eqref{eq:gradl} that $\grad\mc{L}(\Lad,v) = 0$ 
is equivalent to $\Lad = (\l v, A v\r, \l v, B v\r)$ and 
\begin{equation} \label{eqstat}
    ((A - \lad)^2 + (B - \mu )^2) v =  \xi v\,,
\end{equation}
for some constant $\xi \ge 0$. We expand $v = \sum_{i} c_i v_i$ by the common eigenvectors $\{v_i\}_{i = 1}^n$ with the eigenvalue pairs $\{\Lad_i\}_{i = 1}^n$, and find that the equation \eqref{eqstat} reduces to
\begin{equation}  \label{eqstat2}
    \norm{\Lad_i - \Lad}^2 c_i = \xi c_i\,, \q \text{for}\ i = 1,\ldots,n\,,
\end{equation}
 where $\Lad = (\l v, A v\r, \l v, B v\r) = \sum_i |c_i|^2 \Lad_i \in {\rm conv}(\si(A,B))$. Then, it readily follows that $\norm{\Lad_i - \Lad}^2 = \xi$ for any $i$ with $c_i \neq 0$. The proof is complete. 
\end{proof}

From Proposition \ref{lem:statpoint} above, we see that
for $(A,B) \in {\rm F}_{\rm com}$, the stationary points of \eqref{eq:sub-problem} are essentially characterized by the set $\Delta_{\rm sta}$ in \eqref{eqsimstat}, which always includes the nodes $e_i$ and their averages $(e_i + e_j)/2$, while the full analytical characterization of $\Delta_{\rm sta}$ would rely on the distribution of the eigenvalue pairs $\{\Lad_i\}$ of $(A,B)$. The following result shows that in the commuting case, the local minimizers of \eqref{eq:sub-problem} are exactly given by eigenvalue pairs and the common 
eigenvectors. 

\begin{proposition} \label{prob:characmini}
For a given $(A,B) \in {\rm F}_{\rm com}$, 
$(\Lad,v)\in \R^2 \t \S^{n-1}$ is a local minimizer of \eqref{eq:sub-problem} if and only if $\Lad$ is an eigenvalue pair and $v$ is the associated common eigenvector. 
\end{proposition}

\begin{proof}
It suffices to prove the direction $(\Rightarrow)$. Since $(\Lad,v)$ is a local minimizer, we have $\grad \mc{L} (\Lad,v) = 0$ and $\hess \mc{L}(\Lad,v) \succeq 0$ on $\R^2 \t T_v \S^{n-1}$. It follows that the Schur complement of $2 I_2$ in $\hess \mc{L}(\Lad,v)$ is also positive semidefinite, that is, 
\begin{equation} \label{eq:posschur}
    \hess_v\mc{L}(\Lad,v) - 2^{-1} W W^T \succeq 0\,,\ \text{on}\ T_v \S^{n-1}\,.
\end{equation} 
Suppose that $(\Lad,v)$ is not an eigenvalue pair and the common eigenvector of $(A,B)$. From \eqref{eq:repsta}, $\Lad = \sum_{i \in \mc{I}} p_i \Lad_i$ is a convex combination of some non-identical $\{\Lad_i\}_{i \in \mc{I}}$ with $p_i \neq 0$ for $i \in \mc{I}$ and $v = \sum_{i \in \mc{I}} c_i v_i$ with $c_i = \pm \sqrt{p_i}$
(otherwise, $\Lad = \Lad_i$ for some $i$ and $v$ is the common eigenvector). Moreover, there holds $\norm{\Lad - \Lad_i}^2 = \xi$ for $i \in \mc{I}$ and some constant $\xi$. Then, for any $q = \sum_{i \in \mc{I}} b_i v_i \in T_v \S^{n-1}$ with $\norm{q} = 1$, we can compute, by \eqref{eq:rieman_Lv} and \eqref{eq:posschur}, 
\begin{align} \label{auxeqmini}
   & \big\l q, \big(\hess_v\mc{L}(\Lad,v) - 2^{-1} W W^T\big) q \big\r \notag \\
 = & \l q, \na^2_v \mc{L}(\Lad,v) q\r - \l v, \na^2_v \mc{L}(\Lad,v) v\r - 8 (\l q, A v\r^2 + \l q, B v\r^2) \notag \\
 = & 2 \xi - 2 \xi - 8 (\l q, A v\r^2 + \l q, B v\r^2) \ge 0\,.
\end{align} 
Noting that $\{\Lad_i = (\lad_i,\mu_i)\}_{i \in \mc{I}}$ are not identical, either $A v$ or $B v$ is not collinear with $v$. Thus, we can choose $q \in T_v \S^{n-1}$ such that $\l q, A v\r^2 + \l q, B v\r^2 > 0$, which contradicts with \eqref{auxeqmini} and completes the proof.
\end{proof}}

However, the following observation shows that any small perturbation of commuting symmetric matrices could eliminate all common eigenvectors.

\begin{lemma} \label{lem:opendensenocommon}
The set ${\rm M}: = \{(A,B)\in \sn \t \sn\,;\ (A,B)\ \text{has \emph{no} common eigenvectors}\}$ is an open dense subset of the space $\sn \t \sn$. 
\end{lemma}
\begin{proof}
Let $\{\lad_i\}_{i \in \mc{I}_A}$ and $\{\mu_i\}_{i \in \mc{I}_B}$ be the distinct eigenvalues of $A$ and $B$ with associated eigenspaces $E_{i,A}$ and 
$E_{i,B}$, respectively. We define $S_{i,A}: = \mathbb{S}^{n-1} \cap E_{i,A}$ and $S_{i,B}: = \mathbb{S}^{n-1} \cap E_{i,B}$, and see that both $\cup_{i \in \mc{I}_A} S_{i,A}$ and $\cup_{i \in \mc{I}_B} S_{i,B}$ are closed subsets of $\mathbb{S}^{n-1}$, and 
$(A,B)$ has no common eigenvectors if and only if the set $(\cup_{i \in \mc{I}_A} S_{i,A}) \cap (\cup_{i \in \mc{I}_B} S_{i,B})$ is empty. It follows that there exist open neighbourhoods of $(\cup_{i \in \mc{I}_A} S_{i,A})$ and  $(\cup_{i \in \mc{I}_B} S_{i,B})$ in $\mathbb{S}^{n-1}$ with empty intersection.
Then, by the perturbation theory, we can conclude that the set ${\rm M}$ is open in $\sn \t \sn$.  
We next show that ${\rm M}$ is also dense. It suffices to prove that for $A,B \in \sn$ having common eigenvectors and for small enough $\ep > 0$, we can find $A_\ep$, $B_\ep \in \sn$ with $\norm{A - A_\ep} \le \ep$, $\norm{B - B_\ep} \le \ep$ such that $A_\ep$, $B_\ep$ have no common eigenvectors.
For this, by perturbing the eigenvalues of $A$ and $B$, without loss of generality, we assume that both $A$ and $B$ have only simple eigenvalues. We let $A_\ep = A$ and next construct $B_\ep$. Consider the spectral decompositions for $A$ and $B$: $A w_i = \lad_i w_i$ and $B v_i = \mu_i v_i$ for $1 \le i\le n$. Then, for a given $\ep > 0$, we choose an orthogonal matrix $U_\ep$ satisfying $U_\ep v_i \neq v_i$ for any $i$ and $\norm{U_\ep - I} \le \ep$, and define the matrix $B_\ep$ by $B_{\ep} U_\ep  V =  U_\ep V  D$, 
where $V = [v_1,\ldots,v_n]$ and $D = {\rm diag}(\mu_1,\ldots,\mu_n)$. When $\ep$ is small enough, the condition $\norm{U_\ep - I} \le \ep$ implies $U_\ep v_ i \neq v_j, w_k$ for all $1 \le  i ,j ,k \le n$. Thus, $A_\ep$ and $B_\ep$ have no common eigenvectors, and the proof is complete. 
\end{proof}

Although an arbitrary pair of symmetric matrices may not have any common eigenvectors, the functional $\mc{L}(\Lad,v)$ can still admit many local minima. Indeed, recalling the decomposition \eqref{eq:pertburab} for almost commuting matrices $(A,B)$, let $E_*$ be a common eigenspace \eqref{defcomm} of $(A_*,B_*)$ associated with some eigenvalue pair $\Lad_* := (\lad_*,\mu_*)$:
\begin{equation}\label{comoninst}
    E_* := \{v \in \R^n\,;\ A_* v = \lad_* v\,,\ B_* v = \mu_* v\}\,.
\end{equation}
We shall prove in Theorem \ref{them:stab} below that when $\norm{[A,B]} \ll 1$, there always exists a local minimizer to $\mc{L}$ in some neighborhood of $(\Lad_*, S_*)$ in $\R^2  \t \mathbb{S}^{n-1}$, where $S_* : = \mathbb{S}^{n-1} \cap E_*$. For convenience, we denote by 
$\mc{P}_*$ and $\mc{P}_*^\perp$ the orthogonal projections to the subspace $E_*$ and its orthogonal complement $E_*^\perp$, respectively. The following lemma is useful in the sequel.

\begin{lemma} \label{lem:projection}
Let the set $S_*$ and the projections $\mc{P}_*$ and $\mc{P}_*^\perp$
be defined as above. For $v \in \mathbb{S}^{n-1}$, it holds that 
\begin{align} \label{eq:claimone}
    \dist(v,S_*)^2: = \min_{w \in S_*} \norm{v - w}^2  = (1 - \norm{\mc{P}_*(v)})^2 + \norm{\mc{P}_*^\perp(v)}^2 = 2 - 2 \norm{\mc{P}_*(v)}
    \,,
\end{align}
where the minimum value is attained at $v_* = \mc{P}_*(v)/\norm{\mc{P}_*(v)}$. Similarly, let $S_{*,\ep}$ be the $\ep$-neighborhood of the set $S_*$: $S_{*,\ep}: = \{v \in \mathbb{S}^{n-1}\,;\ \dist(v,S_*) \le \sqrt{\ep}\}$. Then, there holds, for $v \in \mathbb{S}^{n-1}$, 
\begin{align} \label{eq:claimtwo}
    \dist(v,S_{*,\ep})^2 = \min_{w \in S_{*,\ep}} \norm{v - w}^2 = \norm{v- \mc{P}_{\ep}(v)}^2\,, 
\end{align}
\mb{where $\mc{P}_{\ep}(v)$ is
the unique minimizer to $\dist(v,S_{*,\ep})$ given by}
\begin{align} \label{eq:claimtwo_2}
    \mc{P}_{\ep}(v) = \theta \frac{\mc{P}_*(v)}{\norm{\mc{P}_*(v)}} + \sqrt{1 - \theta^2} \frac{\mc{P}^\perp_*(v)}{\norm{\mc{P}^\perp_*(v)}} \q \text{with} \q   \theta = \max \left\{1 - \frac{\ep}{2}, \norm{\mc{P}_*(v)}\right\}. 
\end{align}
\end{lemma}

\begin{proof}
The first claim \eqref{eq:claimone} directly follows from the definition, so we only show the second one \eqref{eq:claimtwo}. For this, we note from \eqref{eq:claimone}  that $v \in S_{*,\ep}$ is equivalent to $\norm{\mc{P}_*(v)} \ge 1 - \ep/2$. Hence the statements \eqref{eq:claimtwo} and \eqref{eq:claimtwo_2} clearly hold for $v \in S_{*,\ep}$. Then we consider the case where $v \in \mathbb{S}^{n-1}\backslash S_{*,\ep}$, which gives $\theta = 1 - \ep/2$. It is easy to see 
\begin{align}\label{auxeq_proj}
\norm{v - w}^2  = \norm{\mc{P}_*(v - w)}^2 +  \norm{\mc{P}^\perp_*(v - w)}^2 \ge \big|\norm{\mc{P}_*(v)} - \norm{\mc{P}_*(w)} \big|^2 + \big|\norm{\mc{P}^\perp_*(v)} - \norm{\mc{P}^\perp_*(w)}\big|^2. 
\end{align}
By elementary analysis, we have that the function $\big|\norm{\mc{P}_*(v)} - x \big|^2 + \big|\norm{\mc{P}^\perp_*(v)} - \sqrt{1-x^2}\big|^2$ is monotone increasing in $x \in [1-\ep/2,1]$ for $v \in \mathbb{S}^{n-1}\backslash S_{*,\ep}$. Therefore, 
it follows from \eqref{eq:claimtwo_2} and \eqref{auxeq_proj} that 
\begin{equation*}
      \dist(v,S_{*,\ep})^2 \ge \big|\norm{\mc{P}_*(v)} - \theta \big|^2 + \big|\norm{\mc{P}^\perp_*(v)} - \sqrt{1-\theta^2}\big|^2 = \norm{v - \mc{P}_\ep(v)}^2\,.
\end{equation*}
The proof is complete. 
\end{proof}

\begin{remark}
One can also define the distance function $\widetilde{{\rm dist}}(v, E) := \inf_{w \in E} d(v, w)$ for $E \subset \S^{n-1}$, by using the geodesic distance $d(v,w) := \arccos \l v , w \r$
on $\S^{n-1}$, which has been investigated in \cite{ferreira2013projections,ferreira2014concepts}. It is easy to see that $\dist(v, E)^2 = 2 - 2 \cos \widetilde{{\rm dist}}(v, E)$ if $0 \le \widetilde{{\rm dist}}(v, E) \le \pi$, which implies that
 $\dist$ and $\widetilde{{\rm dist}}$ are equivalent, and $\mc{P}_\ep(v)$ in \eqref{eq:claimtwo_2} is also the unique minimizer to
$\widetilde{\dist}(v,S_{*,\ep})$. However, the interested set $S_{*,\ep}$ is not 
geodesically convex by \cite[Proposition 2]{ferreira2014concepts}. Therefore, the results in \cite{ferreira2013projections,ferreira2014concepts} do not directly apply to our case and we choose to focus on the function $\dist(\dd,\dd)$ in \eqref{eq:claimtwo} in this work. 
\end{remark}

\begin{theorem} \label{them:stab}
Suppose that $(A,B)$ are almost commuting matrices with decomposition \eqref{eq:pertburab}, \mb{and denote by $\d(A_*, B_*)$ the eigenvalue gap \eqref{eqgapp} for commuting matrices $(A_*,B_*)$ in \eqref{eq:pertburab}}. Also, let the eigenvalue pair $\Lad_* = (\lad_*,\mu_*)$ with the space $E_*$ and the set $S_*$ be defined as before Lemma \ref{lem:projection}. We assume 
\begin{equation*}
\norm{\Delta A} + \norm{\Delta B} \le \frac{1}{8} \mb{\d(A_*,B_*)},
\end{equation*}
and define constants
\begin{equation} \label{eq:range}
    \eta := \mb{\frac{15}{4}} (\norm{\Delta A} + \norm{\Delta B}) \q \text{and}\q \ep := \frac{128}{3} \left( \frac{\norm{\Delta A} + \norm{\Delta B}}{\mb{\d(A_*,B_*)}} \right)^2\,.
\end{equation}
Then, it holds that \eqref{eq:sub-problem} admits a local minimizer in the domain:
 \begin{equation} \label{eq:local_domain}
    \mb{D_{\eta,\ep} := \big\{(\Lad, v)\in \R^2 \t \S^{n-1}\,;\ \norm{\Lad - \Lad_*} < \eta,\, \dist(v,S_*) < \sqrt{\ep}\big\}\,.}
 \end{equation}
\end{theorem}

\begin{proof}
By Lemma \ref{lem:projection}, we write, for $v \in \mathbb{S}^{n-1}$, 
\begin{align} \label{preeq_1}
v = c_* v_* + \sum_{j = 1}^{m} c_j v_j\,,\q |c_*|^2 + \sum_{j = 1}^{m}|c_j|^2 = 1\,,
\end{align}
and then have 
\begin{align} \label{preeq_2}
    \dist (v, S_*)^2 := |c_* - 1|^2 + \sum_{i = 1}^m |c_i|^2  = 2 -2 c_*\,,
\end{align}
where $c_* = \norm{\mc{P}_*(v)}$, $v_* = \mc{P}_*(v)/c_* \in S_*$, and $\{v_j\}_{j = 1}^m$ are common orthonormal eigenvectors of $(A_*,B_*)$ spanning  $E_*^\perp$. Note that for any $\ep$ and $\eta$, the set $\overline{D_{\eta,\ep}}$ is compact, which implies that  $\inf_{\overline{D_{\eta,\ep}}} \mc{L}(\Lad,v)$ admits a minimizer. We shall prove that for small enough $\norm{\Delta A}+\norm{\Delta B}$, there exists $\ep$ and $\eta$, depending on $\norm{\Delta A} + \norm{\Delta B}$ and satisfying $\ep, \eta \to 0$ as $\norm{\Delta A} + \norm{\Delta B} \to 0$, such that the minimizer of $\inf_{\overline{D_{\eta,\ep}}} \mc{L}(\Lad,v)$ is in the open set $D_{\eta,\ep}$ and hence also a local minimizer of \eqref{eq:sub-problem}. For this, we define a constant
\begin{align} \label{eq:subopt}
   L_0 := & \inf_{v \in S_*} \norm{(A - \lad_*)v}^2 + \norm{(B - \mu_*)v}^2 \notag  \\  = & \inf_{v \in S_*} \norm{\Delta A v}^2 + \norm{\Delta B v}^2 \le \norm{\Delta A}^2 + \norm{\Delta B}^2\,,
\end{align}
and, without loss of generality, assume $L_0 > 0$. Indeed, if $L_0 = 0$ holds with a minimizer $v \in S_*$, then we readily have $\mc{L}(\Lad_*, v) = 0$, and the point $(\Lad_*, v)$ is the desired local minimizer of \eqref{eq:sub-problem} in $D_{\eta,\ep}$.

We next construct constants $\ep$ and $\eta$ such that
\begin{align} \label{eq:claim}
    \inf_{\overline{D_{\eta,\ep}} \backslash D_{\eta,\ep}} \mc{L}(\Lad,v) > \norm{\Delta A}^2 + \norm{\Delta B}^2\,,
\end{align}
which, by \eqref{eq:subopt}, yields 
$  \inf_{\overline{D_{\eta,\ep}} \backslash D_{\eta,\ep}} \mc{L}(\Lad,v) >  L_0$. It follows that $\arg\min _{\overline{D_{\eta,\ep}}} \mc{L}(\Lad,v) \in D_{\eta,\ep}$ 
and hence completes the proof. \mb{We first have $\overline{D_{\eta,\ep}}\backslash D_{\eta,\ep} \subset D_1 \cup D_2$ with 
\begin{align*}
    & D_1 := \big\{(\Lad, v) \in \R^2 \t \S^{n-1}\,;\ 
    \norm{\Lad - \Lad_*} = \eta\,,\ \dist(v,S_*) \le \sqrt{\ep}\big\}\,,\\
    % |\lad - \lad_*| = \eta\,,\ |\mu - \mu_*| \le \eta\} \t \left\{v \in \mathbb{S}^{n-1}\,; \dist(v,S_*) \le \sqrt{\ep} \right\},\\
    % &  D_2  := \{(\lad,\mu)\,;\  |\lad - \lad_*| \le \eta\,,\  |\mu - \mu_*| = \eta\} \t \left\{v \in \mathbb{S}^{n-1}\,; \dist(v,S_*) \le \sqrt{\ep} \right\},\\
    % &D_2  :=  \{(\lad,\mu)\,;\  |\lad - \lad_*| \le \eta\,,\  |\mu - \mu_*| \le \eta\} \t \left\{v \in \mathbb{S}^{n-1}\,; \dist(v,S_*) = \sqrt{\ep} \right\}.\\
 & D_2  :=  \big\{(\Lad, v) \in \R^2 \t \S^{n-1}\,;\ 
    \norm{\Lad - \Lad_*} \le \eta\,,\ \dist(v,S_*) = \sqrt{\ep}\big\}\,.
\end{align*}
A simple estimate implies 
\begin{align} \label{auxineq_dd}
   \sqrt{2 \mc{L}(\Lad,v)} &\ge \norm{(A_* - \lad)v} + \norm{(B_* - \mu)v} - \norm{\De A} - \norm{\De B}\,.
\end{align}
We now estimate $\mc{L}(\Lad,v)$ on $D_1$. By \eqref{preeq_1} and \eqref{preeq_2}, there holds $c_* \ge 1 - \ep/2$ for $(\Lad,v) \in D_1$, which gives 
\begin{align*}
     \norm{(A_* - \lad)v} + \norm{(B_* - \mu)v} \ge |c_*| \norm{\Lad - \Lad_*} \ge \left(1 - \frac{\ep}{2} \right)\eta\,,
\end{align*}
where we have also used $|\lad - \lad_*| + |\mu - \mu_*| \ge \norm{\Lad - \Lad_*}$. 
If we choose $\ep$ and $\eta$ satisfying
\begin{align} \label{auxeq_coeff_1}
\left(1 - \frac{\ep}{2} \right)\eta > (1 + \sqrt{2}) (\norm{\Delta A} +  \norm{\Delta B})\,,
\end{align}  
it follows from \eqref{auxineq_dd} 
that 
\begin{equation} \label{eq:claim11}
    \inf_{D_1} \mc{L}(\Lad,v) > \norm{\De A}^2 + \norm{\De B}^2\,.
\end{equation}
We next estimate $\mc{L}(\Lad,v)$ on $D_2$. Let $\Lad_j = (\lad_j, \mu_j)$ be the eigenvalue pair of $(A_*,B_*)$ associated with the common eigenvector $v_j$. 
Similarly by \eqref{preeq_1}, we can derive 
 \begin{align} \label{addeq1}
     \norm{(A_* - \lad)v} + \norm{(B_* - \mu)v} & \ge \sqrt{\sum_{j = 1}^{m} |c_j (\lad_j - \lad)|^2} + \sqrt{\sum_{j = 1}^{m} |c_j (\mu_j - \mu)|^2} \notag \\
     & \ge \sqrt{\sum_{j = 1}^{m} |c_j|^2 \norm{\Lad_j - \Lad}^2}\,,
\end{align}
thanks to the elementary inequality $\sqrt{x} + \sqrt{y} \ge \sqrt{x + y}$ for $x,y \ge 0$. For $(\Lad,v) \in D_2$ and small enough $\eta$, it is clear that $\sum_{j = 1}^{m} |c_j|^2 = 1 - (1 - \ep/2)^2$ and 
\begin{equation*}
     \norm{\Lad - \Lad_j} \ge \norm{\Lad_j - \Lad_*} - \norm{\Lad - \Lad_*} \ge \d(A_*, B_*) - \eta\,.
\end{equation*}
Then, if follows from \eqref{addeq1} that 
\begin{align*}
      \norm{(A_* - \lad)v} + \norm{(B_* - \mu)v} \ge \sqrt{\sum_{j = 1}^{m} |c_j|^2 \norm{\Lad_j - \Lad}^2} \ge (\d(A_*, B_*) - \eta) \sqrt{1 - \left(1- \frac{\ep}{2}\right)^2}\,.
\end{align*}
If we choose $\ep$ and $\eta$ such that
\begin{align} \label{auxeq_coeff_2}
   (\d(A_*, B_*) - \eta)\sqrt{1- \left(1 - \frac{\ep}{2}\right)^2} > (1 + \sqrt{2}) (\norm{\Delta A} +  \norm{\Delta B})\,,
\end{align}
then recalling \eqref{auxineq_dd}, we have 
\begin{equation} \label{eq:claim22}
\inf_{D_2} \mc{L}(\Lad,v) > \norm{\De A}^2 + \norm{\De B}^2\,.    
\end{equation}
One can check that in the regime $\norm{\Delta A} + \norm{\Delta B} \le \d(A_*, B_*)/8$, the choice \eqref{eq:range} of parameters $(\ep,\eta)$ satisfies the conditions
\eqref{auxeq_coeff_1} and \eqref{auxeq_coeff_2}. Then, the claim \eqref{eq:claim} follows from \eqref{eq:claim11} and \eqref{eq:claim22} as desired.}
\end{proof}

\begin{remark}
\mb{The prefactors in Theorem \ref{them:stab} (e.g., $\frac{1}{8}$ and $\frac{15}{4}$) are chosen for convenience and can be improved, but the orders in $\norm{\Delta A} + \norm{\Delta B}$ and $\delta(A_*, B_*)$ are optimal by the standard perturbation theory.}  
\end{remark}

\begin{remark}
\mb{In the extreme case where $\si(A_*,B_*)$ is a 
singleton, equivalently, $(A_*, B_*) = (a I, b I)$ for some $a, b \in \R$, we have $S_* = \S^{n-1}$ and $\d(A_*, B_*)$  can be set as $+ \infty$ by convention (note that \eqref{eqgapp} is not well-defined in this case). Then, Theorem \ref{them:stab} is a trivial fact that \eqref{eq:sub-problem} has a local minimum $(\Lad, v) \in \R^2 \t \S^{n - 1}$ with $\Lad$ near the point $(a,b) \in \R^2$. Moreover, we can see that in this case,  \eqref{eq:sub-problem} reduces to $\min_{(\Lad,v) \in \R^2 \t \S^{n-1}} \norm{(\Delta A - \lad) v}^2 + \norm{(\Delta B - \mu) v}^2$ for $\Delta A, \Delta B \in \sn$, which is hard to analyze due to its generality (except $n = 2$, in which case it is essentially a one-dimensional optimization problem; see also Lemma \ref{prop:ob1} below). It is possible to carry out a landscape analysis as in \cite{zhang2022geometric}, but one may have to consider special $\Delta A$ and $\Delta B$ such as rank-one and diagonal matrices.}
\end{remark}

We next discuss generic properties of local minimizers of \eqref{eq:sub-problem} for almost commuting matrices. \mb{Given $(A,B) \in {\rm F}_{\rm com}$, let $\{\Lad_l\}_{l = 1}^L$ ($L \le n$) be its distinct eigenvalue pairs and $\{E^{\Lad_l}_{A,B}\}_{l = 1}^L$ be the associated common eigenspaces \eqref{defcomm}.} Theorem \ref{them:stab} guarantees that \eqref{eq:sub-problem} for the perturbed pair $(A + \Delta A ,B + \Delta B)$ with $\norm{\Delta A} + \norm{\Delta B} \ll 1$ admits a local minimizer in the neighborhood of each closed set $(\Lad_l, S_l)$, where $S_l: = E_{A,B}^{\Lad_l} \cap \mathbb{S}^{n-1}$. Thus, if $L = n$ holds, equivalently, all of the common eigenspaces are one-dimensional, we readily have that $\mc{L}(\Lad,v)$ has at least $n$ local minima, and the associated minimizing vectors $v_j$ are almost orthogonal to each other in the sense that $|\l v_i, v_j\r| \ll 1$ for $i \neq j$. 
We shall see that generically, it is indeed the case. Note that the set of commuting matrices ${\rm F}_{\rm com} \subset \sn \t \sn$ is a complete metric space.

\begin{lemma} \label{lem:open_dense_common}
The subset ${\rm F}_0: = \{ (A,B) \in {\rm F}_{\rm com}\,;\  \dim(E_{A,B}^{\Lad_l}) = 1\,, \ \Lad_l \in \si(A,B)\}$ is open dense in ${\rm F}_{\rm com}$. 
\end{lemma}

\begin{proof}
To show that ${\rm F}_0$ is dense in ${\rm F}_{\rm com}$, it suffices to note that for $(A,B) \in {\rm F}_{\rm com}$, we can always perturb the eigenvalues of $A$ such that the perturbed matrix $\w{A}$ has only simple eigenvalues, which implies $(\w{A},B) \in {\rm F}_0$, while the claim that ${\rm F}_0$ is open in ${\rm F}_{\rm com}$ follows from the standard perturbation result (cf.\,Lemma \ref{lemma:perturbation}). 
\end{proof}

The above lemma is an analog of Lemma \ref{lem:opendensenocommon}, which gives the genericity of $\dim(E_{A,B}^{\Lad_l}) = 1$ in the topological sense. 
It would also be interesting to interpret this property from a probabilistic perspective. We next define a probability measure on the set ${\rm F}_{\rm com}$ of commuting matrices and show that the event ${\rm F}_0$ happens almost surely. Note that ${\rm F}_{\rm com}$ can be parameterized as
\begin{align*} 
    \Phi: (D_1,D_2,U) \in  {\rm diag}_{\le}(n) \t {\rm diag}_{\le}(n) \t \on \longrightarrow (U D_1 U^T, U D_2 U^T) \in {\rm F}_{\rm com}. 
\end{align*}
Our probabilistic model is largely motivated by GOE in random matrix theory. We recall that $X$ is a GOE if diagonal elements $X_{ii} \sim N(0, 1)$ for $1 \le i \le n$ are i.i.d. Gaussian random variables and off-diagonal elements $\{X_{ij}\}_{1 \le i < j \le n}$ are i.i.d. Gaussian $N(0,\frac{1}{2})$ that are independent of $X_{ii}$. We let $U$ be uniformly distributed on ${\rm O}(n)$ and $D_i \in {\rm diag}_{\le}(n)$, $i =1,2$, be random diagonal matrices with diagonal entries independently sampled from the following probability density on the set $\{x_1\le x_2 \le \cdots \le x_n\}$:
\begin{align*}
\vp(x_1,\cdots,x_n) = c_n e^{-\frac{1}{2}(x_1^2 + \cdots + x_n^2)} \prod_{1\le i < j \le n} (x_j - x_i)\,,
\end{align*}
where $c_n$ is the normalization constant. Equivalently, we define the product probability measure: 
\begin{align*}
    \mathbb{P} = \vp(\lad_1,\cdots,\lad_n) d\lad_1 \cdots d \lad_n \t \vp(\mu_1,\cdots,\mu_n) d \mu_1 \cdots d \mu_n \t \mu_{\rm Haar}\,,
\end{align*}
on ${\rm diag}_{\le}(n) \t {\rm diag}_{\le}(n) \t \on$, where $\mu_{\rm Haar}$ denotes the Haar measure on the orthogonal group $\on$. Clearly, the map $\Phi$ is smooth and thus the pushforward measure $\mathbb{P}_{{\rm F}_{\rm com}}: = \mathbb{P} \circ \Phi^{-1}$ on ${\rm F}_{\rm com}$ is well-defined. In what follows, we limit our discussion to the case where ${\rm F}_{\rm com}$ is equipped with the probability $\mathbb{P}_{{\rm F}_{\rm com}}$. One may have noted that $\vp$ is nothing else than the joint distribution of eigenvalues of GOE \cite{anderson2010introduction}. Thus, the pair $(A,B)\in {\rm F}_{\rm com}$ has 
the same marginal distribution as GOE, i.e., $A \sim {\rm GOE}$, $B \sim {\rm GOE}$. To proceed, we recall a basic result that may be due to Wigner and von Neumann \cite{neumann1993verhalten}; see also \cite{tao2012topics,teytel1999rare}. 

\begin{lemma} \label{lem:simple}
The set of symmetric matrices with degenerate eigenvalues is a surface of codimension $2$ in the $n(n-1)/2$-dimensional real vector space $\sn$. Thus, given an absolutely continuous  probability on $\sn$ with respect to Lebesgue measure, a matrix $X \in \sn$
 has all simple eigenvalues almost surely.  
\end{lemma}

We immediately observe from Lemma \ref{lem:simple} that 
\begin{align*}
    \mathbb{P}_{{\rm F}_{\rm com}} ({\rm F}_{\rm com} \backslash {\rm F}_0) & =  \mathbb{P}_{{\rm F}_{\rm com}}(\Lad \in \si(A,B) \ \text{is degenerate})  \\
    & \le \mathbb{P}_{{\rm GOE}}(X\in \sn\ \text{has a degenerate eigenvalue}) = 0\,,
\end{align*}
and the following corollary holds. 

\begin{corollary} \label{lem:high_prob_common}
For a pair of random commuting matrices drawn from the probability $\mathbb{P}_{{\rm F}_{\rm com}}$, 
almost surely we have that all its common eigenspaces are one-dimensional. 
\end{corollary}

In order to combine these observations with Theorem \ref{them:stab}, we recall from \cite[Corollary 2.2]{nguyen2017random} a tail bound of the spectral gap for Wigner matrices (i.e., off-diagonal entries are i.i.d. sub-Gaussian random variables with mean $0$ and variance $1$ and diagonal entries are
independent sub-Gaussians with mean $0$ and variances bounded by $n^{1-o(1)}$), which include GOE as a subclass. 

\begin{lemma} \label{lem:gapwigner}
For real symmetric Wigner matrices $X_n$ with sub-Gaussian entries, there holds 
\begin{align*}
    \min_{1 \le i \le n -1}(\lad_{i+1} - \lad_i) \ge n^{-\frac{3}{2} - o(1)}\,,
\end{align*}
with probability $1 - o(1)$, where $\{\lad_i\}_{i = 1}^n$ are eigenvalues of $X_n$ (counting multiplicity) in an increasing order. 
\end{lemma}

Applying Lemma \ref{lem:gapwigner} to GOE, by Theorem \ref{them:stab}, we can conclude the following result, which shows the desired claim that generically $\mc{L}(\Lad,v)$ has almost orthogonal local minimizing vectors $v_i$.

\begin{theorem} \label{thm:generic_almost}
Let $A,B \in \sn$ be almost commuting matrices with decomposition \eqref{eq:pertburab}, where $(A_*,B_*) \in {\rm F}_{\rm com}$ is sampled from $\mathbb{P}_{{\rm F}_{\rm com}}$. Then, for $\norm{\Delta A} + \norm{\Delta B} \le n^{-3/2 - o(1)}/8$, with probability $1 - o(1)$, the function $\mc{L}(\Lad,v)$ admits $n$ local minima on $\R^2 \t \mathbb{S}^{n-1}$ with associated minimizing vectors $\{v_j\}_{j = 1}^n$ satisfying
\begin{align*}
    |\l v_i, v_j\r| = O\left( \frac{\norm{\Delta A} + \norm{\De B}}{n^{-3/2-o(1)}}\right). 
\end{align*}
\end{theorem}

%%%%%%%%%%%%%%%%%%%
% \mb{We next consider the case where $(A_*, B_*) \in {\rm F}_{\rm com}$  has a degenerate eigenvalue pair, which is a rare event. By exploiting the optimality conditions of \eqref{eq:sub-problem}, we shall show that for generic perturbations $(\Delta A, \Delta B)$, there could exist 
% $\dim(E_*)$ local minimizers with almost orthogonal minimizing vectors to \eqref{eq:sub-problem} for the 
% almost commuting pair $(A_* + \Delta A, B_* + \Delta B)$ in the vicinity of $(\Lad_*, S_*)$ in $\R^2  \t \mathbb{S}^{n-1}$, where $S_* = \mathbb{S}^{n-1} \cap E_*$. Here, $E_*$ for $\Lad_* = (\lad_*,\mu_*)$ is given as in \eqref{comoninst} above.} 

\section{Vector-wise joint diagonalization algorithm}\label{sec:algorithm}

In this section, we will provide a detailed description of the VJD algorithm for solving \eqref{eq:problem}. Following the ideas presented in the introduction, we first summarize its key steps in Algorithm \ref{alg:VJD} below. 

\RestyleAlgo{ruled}
\begin{algorithm}[!htbp]
\caption{VJD algorithm for \eqref{eq:problem}}\label{alg:VJD}
\KwIn{almost commuting matrices $A,B \in \sn$}
\KwOut{orthogonal matrix $U \in \on$; diagonal matrices $D_1, D_2 \in \diag{n}$}

set $(\lad^{(1)},\mu^{(1)},v^{(1)})$ to the minimizer of \eqref{eq:sub-problem}

$V = v^{(1)}$

$w^{(1)} = v^{(1)}$

\For{$j=2,\cdots,n$}{

set $(\lad^{(j)},\mu^{(j)},v^{(j)})$ to the minimizer of  \eqref{eq:problem with inequality constraint} with inputs $\{w^{(i)}\}_{ i = 1}^{j-1}$

$V = [V, v^{(j)}]$

compute $w^{(j)}$ by orthogonalizing $\{w^{(1)},\ldots,w^{(j-1)}, v^{(j)}\}$ via Gram-Schmidt process 

    % Randomly generate $v_0\in\mathbb{S}^{n-1}$\;
    % \eIf{$k=1$}{
    %     Solve \eqref{eq:sub-problem} using Algorithm~\ref{alg:precondition} resulting in $v_1$\;
    %     $V \gets v_1$, $P \gets v_1$\;
    % }{
    %     Solve \eqref{eq:problem with inequality constraint'} using Algorithm~\ref{alg:for ineq} resulting in $v_k$\;
    %     $\beta_k \gets v_k$\;
    %     \For{$\beta$ in $P$}{
    %         $\beta_k \gets \beta_k - \frac{(v_k,\beta)}{(\beta,\beta)}\beta$\;
    %     }
    %     $V \gets [V,v_k]$, $P \gets [P,\beta_k]$
    % }
}

compute the SVD: $V = U_1 \Sigma U_2$ for solving \eqref{eq:orthogonalize} 

$U = U_1 U_2$, $D_1 = {\rm diag}(\lad^{(1)},\ldots,\lad^{(n)})$, $D_2 = {\rm diag}(\mu^{(1)},\ldots,\mu^{(n)})$

% Singular value decomposition $V=U_1\Sigma U_2$, $U \gets U_1U_2$.
\end{algorithm}
\vspace{- 2 mm}

\begin{remark}
\mb{The introduction of $w^{(j)}$ is conceptually unnecessary and mainly for implementing the projection $\mc{P}_\ep$ \eqref{eq:claimtwo_2} associated with ${\rm span}\{v^{(1)},\ldots, v^{(j)}\}$, which is needed for solving \eqref{eq:problem with inequality constraint}; see Algorithm \ref{alg:for ineq}.}
\end{remark}

The rest of this section is devoted to designing algorithms for the optimization problems \eqref{eq:sub-problem} and \eqref{eq:problem with inequality constraint}. Note that 
the cost functional $\mc{L}(\Lad,v)$ is quadratic and convex with respect to both variables $\Lad \in \R^2$ and $v \in \R^n$, which motivates us to consider the alternating descent framework, namely, successively updating the variables $\Lad$
and $v$ to minimize the cost $\mc{L}$. In particular, we have
\begin{align} \label{eq:update of lambda and mu}
     \arg\min_{\Lad \in \R^2} \mc{L}(\Lad,v) = (\l v, A v \r, \l v, B v \r)\,,
     % \,,\q    \arg\min_{\mu \in \R} \mc{L}(\Lad,v) = \l v, B v \r\,,
\end{align}
and there holds, \mb{for any $v \in \S^{n-1}$ and $\Lad = (\lad,\mu) \in \R^2$,  
\begin{align} \label{eq:wlv}
\w{\mc{L}}(v) \le \mc{L}(\Lad,v)\,,
\end{align}
with 
\begin{equation} \label{eq:wlv0}
   \w{\mc{L}}(v): = \mc{L}(\l v, A v \r,\l v, B v \r,v) = \l v, (A^2 + B^2) v\r - \l v, A v \r^2 - \l v, B v \r^2\,,
\end{equation}
which is a fourth-order polynomial in $v$.} Thus, we can update $\Lad$ explicitly via \eqref{eq:update of lambda and mu}. To update the vector $v \in \S^{n-1}$, we consider Riemannian 
Newton's method with a line search, which takes advantage of the Hessian information of $\mc{L}$ and enables a fast convergence. \mb{More precisely, in the $k$-th step, given $\Lad_k = (\l v_k, A v_k\r, \l v_k, B v_k\r)$ and $v_k \in \S^{n-1}$, we compute the search direction $s_k \in T_{v_k} \S^{n-1}$ by 
\begin{equation} \label{eq:Riemann Newton equation}
  - \hess_v{\mathcal{L}(\Lad_k, v_k)}[s_k] = \grad_v {\mathcal{L}(\Lad_k, v_k)}\,,
\end{equation}
and then update $v_{k+1} = R_{v_k}(\alpha_k s_k)$ and $\Lad_{k + 1} = (\l v_{k+1}, A v_{k+1}\r, \l v_{k+1}, B v_{k+1}\r)$, where the stepsize $\alpha_k$ is determined by a backtracking line search (cf.\,\eqref{eq:new_line_search} below).} However, per our numerical experiments, such an alternating minimization may converge slowly and become less efficient when the matrix size $n$ increases. We also note that $\hess_v \mc{L}$ may have negative eigenvalues so that $s_k$ in \eqref{eq:Riemann Newton equation} is not necessarily a descent direction of $\mc{L}(\Lad,\dd)$, which leads to the potential instability of the algorithm. Our practically efficient algorithm is motivated by the following lemmas.

\begin{lemma} \label{prop:ob1}
Given $A,B \in \sn$, let functions $\mc{L}(\Lad,v)$ on $\R^2 \t \S^{n-1}$ and $\w{\mc{L}}(v)$ on $\S^{n-1}$ be defined in \eqref{eq:sub-problem} and \eqref{eq:wlv0}, respectively. Then, $(\Lad_*, v_*)$ is  a local minimum of $\mc{L}$ if and only if 
$\Lad_* = (\l v_*, A v_*\r, \l v_*, B v_*\r)$ and $v_*$ is a local minimum of $\w{\mc{L}}$. Moreover, we have  
\begin{align} \label{eq:grad_rela}
    \grad\w{\mc{L}}(v) = \mb{\grad_v \mc{L}(\Lad,v)|_{(\Lad,v) = (\l v, A v \r, \l v, B v \r, v)}}\,,
\end{align}
and $\hess \w{\mc{L}}(v)$ is the Schur complement of the block $2 I_2$ of the Hessian $\hess \mc{L}$ \eqref{eq:hessl}, i.e.,
\begin{align} \label{eq:hess_rela}
    \hess \w{\mc{L}}(v) = \mb{\hess_v \mc{L}(\Lad,v)|_{(\Lad,v) = (\l v, A v \r, \l v, B v \r, v)}} - 2^{-1} W W^T\,.
\end{align}
\end{lemma}

\begin{proof}
The Riemannian gradient \eqref{eq:grad_rela} and Hessian \eqref{eq:hess_rela} of $\w{\mc{L}}(v)$ follow from a direct computation by \eqref{eq:riegrad}. For the first claim, let $(\Lad_*, v_*)$ be a local minimum of $\mc{L}$. Then, $\Lad_*$ is a local minimum of the quadratic function $\mc{L}(\dd,v_*)$, which gives
$\Lad_* = (\l v_*, A v_*\r, \l v_*, B v_*\r)$. By \eqref{eq:wlv}, we have, for $(\Lad,v)$ near $(\Lad_*,v_*)$,
\begin{align} \label{auxeq}
    \w{\mc{L}}(v_*) \le \w{\mc{L}}(v) \le \mc{L}(\Lad,v)\,,
\end{align}
i.e., $v_*$ is a local minimum of $\w{\mc{L}}(v)$. The other direction can be proved similarly. 
\end{proof}

Thanks to Lemma \ref{prop:ob1}, optimizing the function $\mc{L}(\Lad,v)$ is equivalent to optimizing the one $\w{\mc{L}}(v)$, since they have exactly the same local minima in the variable $v \in \S^{n-1}$. Noting from \eqref{eq:grad_rela} that 
$\grad_v \mc{L}(\Lad_k, v_k) = \grad \w{\mc{L}}(v_k)$, we will show that in the almost commuting regime $[A,B] \ll 1$, the alternating descent algorithm given above can be regarded as a Riemannian approximate Newton method for solving 
\begin{align} \label{eq:subeq}
    \min_{v \in \S^{n-1}}  \w{\mc{L}}(v) = \l v, (A^2 + B^2) v\r - \l v, A v \r^2 - \l v, B v \r^2\,,
\end{align}
namely, a special case of Algorithm \ref{alg:for ineq} below with $H(v) = \hess_v \mc{L}(\l v, A v\r, \l v, B v\r, v)$. In this work, the approximate Newton method refers to a Newton-type algorithm that employs structurally simple approximate Hessian matrices with easily computed inverses. 

\begin{lemma} \label{prop:ob2}
Given a pair of almost commuting matrices $(A,B)$ with decomposition \eqref{eq:pertburab}, let 
$v_*$ be a local minimum  of the associated $\w{\mc{L}}(v)$ on $\S^{n-1}$. Then, for both $H(v) = \na_v^2 \mc{L}(\Lad,v)$ and $H(v) = \hess_v \mc{L}(\Lad,v)$ with $\Lad = (\l v, A v\r, \l v, B v\r)$, it holds that 
\begin{align} \label{eq:approx_1}
    \norm{\hess \w{\mc{L}}(v) - H(v)} = O(\norm{v - v_*}) + O(\norm{\De A} + \norm{\De B})\,,
\end{align}
for $v$ near $v_*$ and $\norm{\De A} + \norm{\De B}$ small enough. In particular, if $[A,B] = 0$, we have 
\begin{align} \label{eq:approx_2}
    \norm{\hess \w{\mc{L}}(v) - H(v)} = O(\norm{v - v_*})\,.
\end{align}
\end{lemma}

\begin{proof}
It suffices to consider \mb{$H(v) = \na_v^2 \mc{L}(\Lad, v)|_{(\Lad,v) = (\l v, A v \r,\l v ,B v\r,v)}$}, since the other one can be similarly analyzed. By formulas \eqref{eq:gradl} and \eqref{eq:grad_rela},  the optimality condition $\grad \w{\mc{L}}(v_*) = 0$ gives
\begin{align} \label{auxeq_3}
    (H(v_*) - \l v_*, H(v_*) v_*\r I) v_* = 0\,,
\end{align}
which yields, by \eqref{eq:rieman_Lv}, 
\begin{align} \label{auxeq_4}
    M(v_*) := \mb{\hess_v \mc{L}(\Lad,v)|_{(\Lad,v) = (\l v_*, A v_* \r, \l v_*, B v_* \r, v_*)}} = H(v_*) - \l v_*, H(v_*) v_*\r I\,.
\end{align}
For the the desired estimate \eqref{eq:approx_1}, we first observe 
\begin{align*}
    \norm{\hess\w{\mc{L}}(v) - H(v)} \le O(\norm{v - v_*}) + \norm{\hess\w{\mc{L}}(v_*) - H(v_*)}\,,
\end{align*}
by the Lipschitz continuity of $\hess\w{\mc{L}}(v)$ and $H(v)$. Then, recalling from Proposition \ref{prob:characmini} and Lemma \ref{prop:ob1} that the local minima of $\w{\mc{L}}(v)$ for commuting matrices are characterized by their common eigenvectors, we have $H(v_*)  v_* = 0$ and $\hess \w{\mc{L}}(v_*) = M(v_*) = H(v_*)$ for $(A,B) \in {\rm F}_{\rm com}$, by formulas \eqref{eq:wwhess}, \eqref{eq:hess_rela}, and \eqref{auxeq_4}. It follows from 
the stability estimate in Theorem \ref{them:stab} that $\norm{\hess\w{\mc{L}}(v_*) - H(v_*)} = O(\norm{\De A} + \norm{\De B})$. Thus, \eqref{eq:approx_1} and also \eqref{eq:approx_2} hold.
\end{proof}

%%%%%%%%%%%%%%%%%%
% To consider the convergence rate, we first note 
%     \begin{align*}
%         \norm{\hess \mc{L}(v) - H(v)} & \le  \norm{\hess \mc{L}(v) - \hess \mc{L}(v_*) + H(v_*) - H(v)} + \norm{- \hess \mc{L}(v_*) + H(v_*)} \\ 
%         & \le O(\norm{v - v_*}) + O(\norm{\Delta A} + \norm{\Delta B})
%     \end{align*}
%     In the commuting case, we have 
%      \begin{align*}
%         \norm{\hess \mc{L}(v) - H(v)}  \le O(\norm{v - v_*}). 
%     \end{align*}

%%%%%%%%%%%%%%%%% 
% Also notice that it holds
% \begin{equation}\notag
% \begin{split}
%     \nabla\w{L}(v) &= \bracket{(A-(v,Av)I)^2 + (B-(v,Bv)I)^2}v; \\
%     \nabla_v\mathcal{L}(\lambda,\mu,v) &= \bracket{(A-\lambda I)^2 + (B-\mu I)^2}v. 
% \end{split}
% \end{equation}
% After each step of Algorithm~\ref{alg:alt}, $\lambda,\mu$ are updated according to \eqref{eq:update of lambda and mu}, and thus $\nabla\w{L}(v) = \nabla_v\mathcal{L}(\lambda,\mu,v)$. Therefore, we can rewrite the Newton equation \eqref{eq:Riemann Newton equation} as 
% \begin{equation}\label{eq:Newton equation}
%     \hess{\mathcal{L}(\lambda,\mu,v)}[u] = \grad{\w{L}(v)}, 
% \end{equation}
% which can be regarded as a preconditioned gradient descent method for solving the problem
% \begin{equation}\label{eq:tilde L}
%     \min_{\specnorm{v}=1}\w{L}(v). 
% \end{equation}

\mb{In what follows, we say that a symmetric matrix $H(v) \in \sn$ is an approximate Riemannian Hessian of $\w{\mc{L}}(v)$ if the property \eqref{eq:approx_1} holds as $\norm{v - v_*} \to 0$ and $\norm{\Delta A} + \norm{\Delta B} \to 0$.} 
In addition to $\hess_v \mc{L}(\Lad,v)$ with $\Lad = (\l v, A v\r, \l v, B v\r)$ from the alternating minimization, Lemma \ref{prop:ob2} also suggests a new choice $\na_v^2 \mc{L}(\Lad,v)$ for approximating  $\hess \w{\mc{L}}(v)$ with the favorable positive semidefinite property. It is easy to see that its restriction on $T_v \S^{n-1}$: $(I_n - v v^T) \na_v^2 \mc{L}(\Lad,v) (I_n - v v^T)$ also satisfies the estimate \eqref{eq:approx_1} and hence can be a good candidate for the approximate Hessian as well. 

\mb{We are now ready to give our approximate Newton method (Algorithm \ref{alg:for ineq}) for solving \eqref{eq:sub-problem} and \eqref{eq:problem with inequality constraint}. We note that the only difference between \eqref{eq:sub-problem} and \eqref{eq:problem with inequality constraint} is the inequality constraint $\dist(v, S_{j-1}) \le \sqrt{\ep}$, which can be easily dealt with by a  projection step, thanks to Lemma \ref{lem:projection}. One may also observe that for $j = 1$, there holds $S_{j - 1} = \S^{n - 1}$ by definition and then \eqref{eq:problem with inequality constraint} reduces to \eqref{eq:sub-problem}. Thus, it suffices to state the algorithm for \eqref{eq:problem with inequality constraint}. For clarity, we recall that given $w_1,\ldots,w_j \in \S^{n-1}$ ($j \ge 1$), 
\begin{align} \label{eq:inter_span}
    S_j := {\rm span}\{w_1,\ldots,w_j\}^\perp \cap \S^{n-1}\,,
\end{align}
and $S_{j,\ep} := \left\{v \in \S^{n-1}\,;\ {\rm dist}(v, S_j) \le \sqrt{\ep}\right\}$ for some $\ep > 0$. If $j = 0$, then $S_j = S_{j,\ep} := \S^{n-1}$. We denote by $\mc{P}_{j,\ep}$ the projection to the minimizer of $\dist(v,S_{j,\ep})$ given in  \eqref{eq:claimtwo_2} and also recall the retraction $R_v$ in \eqref{def:retraction}.}

\RestyleAlgo{ruled}
\begin{algorithm}[!htbp]
\caption{\mb{Projected Riemannian approximate Newton method for \eqref{eq:sub-problem} and \eqref{eq:problem with inequality constraint}}} \label{alg:for ineq}
\KwIn{almost commuting  matrices $A, B \in \sn$; 

\qquad \quad \ \ \mb{vectors $w_1, \ldots, w_{j-1} \in \mathbb{S}^{n-1}$ for $j \ge 1$ ($j = 1$ means no input vectors)}; 

\qquad \quad \ \  parameters $\tau \in (0,1/2)$, $\beta \in (0,1)$, $ \epsilon > 0$, and $ 0< \ep \ll 1$; 

\qquad \quad \ \  approximate Hessian matrix $H(\dd) \in \sn$} 

\KwOut{a triplet $(\lad,\mu,v) \in \R^2 \t \S^{n-1}$} 

generate random initial vector $v_0 \in S_{j - 1}$ with $S_{j - 1}$ defined by $\{w_1, \ldots, w_{j-1}\}$ as in \eqref{eq:inter_span}

$k = 0$ 

% $v \gets v_0\in\mathbb{S}^{n-1}$, $\lambda \gets (v,Av)$, $\mu \gets (v,Bv)$\;
\While{$\norm{\grad\w{\mc{L}}(v_k)} > \epsilon$}{
compute the search direction $s_k \in T_v \mathbb{S}^{n-1}$ by solving 
\begin{equation} \label{eq:quasi Newton equation}
  - H(v_k)[s_k] = \grad{\w{\mc{L}}}(v_k)
\end{equation}

compute the stepsize $\alpha_k$ by the backtracking line search:
\begin{align} \label{eq:new_line_search}
     \alpha_k := \max\left\{\beta^j\,;\ \w{\mc{L}}(R_{v_k}(\beta^{j}s_k)) \le \w{\mc{L}}(v_k) + \tau \beta^j \l \grad \w{\mc{L}}(v_k), s_k\r \right\}
\end{align}

$v_{k+1} = \mc{P}_{j,\ep} R_{v_k}(\alpha_k s_k)$

$k = k+1$

}

$v = v_k \in \S^{n-1}$, $\lad = \l v_k, A v_k\r$, $\mu = \l v_k, B v_k\r$
\end{algorithm}
\vspace{2 mm}

We end this section with several remarks. First, by Theorem \ref{thm:generic_almost}, the parameter $\ep$ in Algorithm \ref{alg:for ineq} can be chosen of the order $(\norm{\De A} + \norm{\De B})/n^{-3/2}$ or larger. Second, 
by the above discussion, it is clear that Algorithm \ref{alg:for ineq} with $H(v) = \hess_v \mc{L}(\l v, A v\r, \l v, B v\r,v)$ gives the alternating minimization for \eqref{eq:sub-problem} and \eqref{eq:problem with inequality constraint}, while for $H(v) = \hess \w{\mc{L}}(v)$, Algorithm \ref{alg:for ineq} is nothing else than the standard Riemannian Newton's method. Moreover, combining the abstract VJD scheme in Algorithm \ref{alg:VJD} with Algorithm \ref{alg:for ineq} gives the one used in practice. 

Finally, for the step \eqref{eq:quasi Newton equation}, if the approximate Hessian maps $T_v \S^{n-1}$ to $T_v \S^{n-1}$, we can simply take the Moore Penrose generalized inverse of $H$ to find $s_k = - H(v_k)^\dag \grad \w{\mc{L}}(v_k)$.  
But if $H(v)$ does not have $T_v \S^{n-1}$ as an invariant subspace, say, $H(v) = \na_v^2 \mc{L}(\l v, A v\r, \l v, B v\r,v)$, then the equation \eqref{eq:quasi Newton equation} may not admit a solution in $T_v \S^{n-1}$. In this case, we consider the least-squares solution to \eqref{eq:quasi Newton equation}. See Section \ref{subsec:compare H} for more computational details, where we numerically test the aforementioned candidates of the approximate Hessian $H(v)$.

%%%%%%%%%%%%%%%
% We also emphasize that Algorithm \ref{alg:VJD} can be directly extended to jointly diagonalize almost commuting $m$-tuples of symmetric matrices. 

% \vspace{2mm}
% \RestyleAlgo{ruled}
% \begin{algorithm}[!htbp]
% \caption{Projecting method for \eqref{eq:problem with inequality constraint'}}\label{alg:for ineq}
% $v \gets v_0\in\mathbb{S}^{n-1}$\;
% \While{$\specnorm{\grad{\w{L}(v)}}>\epsilon$}{
%     Solve the equation
%     \begin{equation}\notag
%         H(v)u = \grad{\w{L}(v)}
%     \end{equation}
%     for unknown $u$\;
%     Choose a step length $\tau$ using Algorithm~\ref{alg:backtracking}, $v \gets (v-\tau u)/\specnorm{v-\tau u}$\;
%     \eIf{$\abs{(v,\beta_j)}\le\delta,j=1,\cdots,k$}{pass}{
%         Let $\beta=\arg\max_j\abs{(\beta_j,v)}$, $v^\perp \gets (I-P_kP_k^T)v$, solve $\theta$ from \eqref{eq:theta}, $v \gets \frac{v^\perp + \theta(v-v^\perp)}{\specnorm{v^\perp + \theta(v-v^\perp)}}$\;
%     }
%     $\lambda \gets (v,Av), \mu \gets (v,Bv)$\;
% }
% \end{algorithm}
% \vspace{2mm}

%%%%%%%%%%%%%
% Algorithm~\ref{alg:VJD} can be directly generalized to jointly diagonalizing $m$ real symmetric matrices $A_1,\cdots,A_m$, which are commuting or almost commuting. We only need to modify the objective function to be 
% \begin{equation}\notag
%     \mathcal{L}(\lambda_1,\cdots,\lambda_m,v) = \sum_{i=1}^m\specnorm{(A_i-\lambda_iI)v}^2.
% \end{equation}
% In Section~\ref{subsec:ICA}, the generalized algorithm is applied for independent component analysis (ICA).  

\section{Error analysis}\label{sec:error analysis}
This section is devoted to the error analysis for our VJD algorithm. In Section \ref{subsec:converge}, we establish the global convergence of Algorithm \ref{alg:for ineq} for \eqref{eq:sub-problem} (i.e., without the inequality constraint), while in Section \ref{subsec:stable}, we discuss the numerical stability of Algorithm \ref{alg:VJD} with respect to noises and numerical errors. 

\subsection{Convergence of Riemannian approximate Newton method} \label{subsec:converge}
Before stating our main theorem, we recall some preliminaries. First, by \cite[Lemma 6]{ring2012optimization}, there exist $a_0 > 0$, $a_1 > 0$, and $\d_{a_0,a_1} > 0$ such that for any $v \in \S^{n-1}$ and $p \in T_v \S^{n-1}$ with $\norm{p} \le \d_{a_0,a_1}$, 
\begin{align} \label{eq:localest}
a_0 \norm{p} \le d(v, R_v(p)) \le a_1 \norm{p}\,,
\end{align}
where $d(\dd,\dd)$ is the geodesic distance on $\S^{n-1}$. We define the following constant for later use:
\begin{align} \label{eq:localconst}
    K := \sup\left\{\frac{d(R_v(p), R_v(q)) }{\norm{p-q}}\,;\ p, q \in T_v \S^{n-1}\,,\ \norm{p} \le \d_{a_0,a_1}\,,\ \norm{q} \le \d_{a_0,a_1}  \right\}. 
\end{align}
We denote by $P_{v,w}$ the parallel transport of tangent vectors at $v$ to the tangent vectors at $w$ along the geodesic curve connecting points $v$ and $w$; see \cite{boumal2022intromanifolds} for the precise definition of $P_{v,w}$. We also recall the following expansion for the Riemannian gradient of a smooth function $f$ on $\S^{n-1}$ \cite{ferreira2002kantorovich,li2019convergence}: for $w$ near $v$, 
   \begin{align} \label{eq:gradient_exp}
    \grad f (w) = P_{v, w} \grad f(v) + P_{v, w} \hess f(v)[R_{v}^{-1}w] + O(\norm{R_{v}^{-1}w}^2)\,.
\end{align}

\begin{theorem} \label{thm:glob_conv}
Let $H(v): T_v \S^{n-1} \to T_v \S^{n-1}$ be an positive semidefinite approximate Hessian of $\w{\mc{L}}$ satisfying \eqref{eq:approx_1}. Suppose that $v_*$ is an accumulation point of the iterative sequence $\{v_k\}$ generated by Algorithm \ref{alg:for ineq} for \eqref{eq:sub-problem}; and that $H(v_k)$ for all $k$ and $H(v_*)$ are non-singular operators on the tangent space. Then, $v_*$ is a stationary point of \eqref{eq:sub-problem}, and it holds that for $0 \neq  \norm{[A,B]} \ll 1$, the sequence $\{v_k\}$ converges to $v_*$ linearly with the convergence rate:
\begin{equation*}
    \mb{\limsup_{k \to \infty} \frac{d(v_{k+1},v_*)}{d(v_k,v_*)} = O(\norm{[A,B]}^{1/2})\,.}
\end{equation*}
In the commuting case $[A,B] = 0$, we have that $\{v_k\}$ converges to $v_*$ quadratically:
\begin{equation*}
    \mb{\limsup_{k \to \infty} \frac{d(v_{k+1},v_*)}{d(v_k,v_*)^2} = 0\,.}
\end{equation*}
\end{theorem}

\begin{proof}
In this proof, we always regard $H(v)$ as a linear operator on the tangent space $T_v \S^{n-1}$. By abuse of notation, we define the norm for any $T: T_v \S^{n-1} \to T_v \S^{n-1}$ by $\norm{T} = \sup\{\norm{T v}\,;\ v \in T_v \S^{n-1}\,,\ \norm{v} = 1\}$. Since $H(v_*)$ is non-singular and $H(v)$ is continuous, there exists a neighborhood $U$ of $v_*$ such that $H(v)$ for $v \in U$ is positive definite on $T_v \S^{n-1}$ and satisfies
\begin{align} \label{eq:localinver}
    \norm{H(v)^{-1}} \le 2 \norm{H(v_*)^{-1}}\,,
\end{align}
which implies that we can find a constant $\si > 0$ such that $H(v_k)^{-1} \ge \si I$ on $T_{v_k} \S^{n-1}$.  It follows that 
% \begin{equation} \label{eq:sknorm}
%       \si \norm{\grad \w{\mc{L}}(v_k)}  \le \norm{s_k} \le \h{\si} \norm{\grad \w{\mc{L}}(v_k)}\,,
% \end{equation}
% and 
\begin{align} \label{eq:descent}
    \l \grad \w{\mc{L}}(v_k), s_k \r = - \l \grad \w{\mc{L}}(v_k), H(v_k)^{-1} \grad \w{\mc{L}}(v_k)  \r \le - \si \norm{\grad \w{\mc{L}}(v_k)}^2\,.
\end{align}
Then by \eqref{eq:descent} and the line search condition \eqref{eq:new_line_search}, we see that the stepsize $\alpha_k$ is well defined. 

\emph{Step 1.} We first prove that the accumulation point $v_*$ is also the stationary point, i.e., $\grad \w{\mc{L}}(v_*) = 0$.  Again by  \eqref{eq:new_line_search}, we have 
\begin{align} \label{eq:decrease}
    \w{\mc{L}}(v_k) - \w{\mc{L}}(v_{k+1}) \ge - \tau \alpha_k \l \grad \w{\mc{L}}(v_k), s_k \r \ge \alpha_k \si \tau \norm{\grad \w{\mc{L}}(v_k)}^2 > 0\,,
\end{align}
which means that $\w{\mc{L}}(v_k)$ is strictly decreasing. By \eqref{eq:decrease}, if $\liminf_{k \to \infty} \alpha_k = \mu > 0$ holds, then we find
\begin{align*}
     \sum_{k = 1}^\infty \norm{\grad \w{\mc{L}}(v_k)}^2 < + \infty\,,
\end{align*}
and $\norm{\grad \w{\mc{L}}(v_k)} \to 0$ as $k \to \infty$, which readily gives $\norm{\grad \w{\mc{L}}(v_*)} = 0$. If not, i.e., 
% If not, then noting from \eqref{eq:sknorm} that $s_k$ is bounded, passing to a subsequence, we let $s_k \to s_*$.
% By the assumption 
$\liminf_{k \to \infty} \alpha_k = 0$ holds, given any $s \in \mathbb{N}$, we can take a subsequence $k_j$ such that $\alpha_{k_j} < \beta^{-s}$ for $j$ large enough.
Hence, $\beta^{-s}$ does not satisfy the linear search condition \eqref{eq:new_line_search}, that is, 
\begin{equation} \label{eq:violinesearch}
     \w{\mc{L}}(R_{v_{k_j}}(\beta^{-s} s_{k_j})) > \w{\mc{L}}(v_{k_j}) + \tau \beta^{-s} \l \grad \w{\mc{L}}(v_{k_j}), s_{k_j} \r\,.
\end{equation}
Without loss of generality, we assume that there holds $v_{k_j} \to v_*$ and $s_{k_j} \to s_* = - H(v_*)^{-1} \grad \w{\mc{L}}(v_*)$. 
Taking the limit $j \to \infty$ in \eqref{eq:violinesearch} gives 
\begin{equation*}
     \w{\mc{L}}(R_{v_{*}}(\beta^{-s} s_{*})) > \w{\mc{L}}(v_{*}) + \tau \beta^{-s} \l \grad \w{\mc{L}}(v_{*}), s_{*} \r\,.
\end{equation*}
Then it is easy to see 
\begin{align*}
    \tau \l \grad \w{\mc{L}}(v_{*}), s_{*} \r  < \liminf_{s \to \infty} \frac{\w{\mc{L}}(R_{v_{*}}(\beta^{-s} s_{*})) - \w{\mc{L}}(v_{*})}{\beta^{-s}} = \l \grad \w{\mc{L}}(v_*), s_* \r\,. 
\end{align*}
Noting $\tau \in (0,1/2)$ in Algorithm \ref{alg:for ineq}, we obtain $\l \grad \w{\mc{L}}(v_*), s_*\r  = 0$ and thus
$\grad \w{\mc{L}}(v_*) = 0$. 

\emph{Step 2.} We next show that there exists a neighborhood $V \subset U$ of $v_*$ such that if $v_k \in V$, then $\alpha_k = 1$. Recalling that $R_v$ is a second-order retraction, we have the following Taylor expansion \cite[Proposition 5.43]{boumal2022intromanifolds}: 
\begin{equation} \label{eq:taylor}
\w{\mc{L}}(R_v(s)) = \w{\mc{L}}(v) + \l \grad \w{\mc{L}}(v), s\r +  \frac{1}{2}\l \hess \w{\mc{L}}(v)[s], s\r +  O(\norm{s}^3)\,,
\end{equation}
for $s \in T_v \S^{n-1}$ with $\norm{s}$ small enough. Then by \eqref{eq:localinver}, for $v$ near $v_*$, it holds that 
\begin{align*}
     s = - H(v)^{-1} \grad \w{\mc{L}}(v) = O(\norm{v - v_*})\,,
\end{align*}
which, along with \eqref{eq:taylor}, yields 
\begin{align} \label{eq:taylorred}
   \w{\mc{L}}(R_v(s)) - \w{\mc{L}}(v) & = \l \grad \w{\mc{L}}(v) ,s \r +  \frac{1}{2}\l \hess \w{\mc{L}}(v)[s], s\r + O(\norm{s}^3) \notag \\
   & = \frac{1}{2} \l \grad \w{\mc{L}}(v) ,s \r + O\left(\norm{\Delta A} + \norm{\Delta B}\right) \norm{v - v_*}^2 + O(\norm{v - v_*}^3)\,,
\end{align}
where we have also used the assumption that \eqref{eq:approx_1} holds. The above estimate \eqref{eq:taylorred} allows us to find a neighborhood $V$ of $v_*$ such that for any $v \in V$, there holds 
\begin{equation*}
    \w{\mc{L}}(R_v(s)) - \w{\mc{L}}(v) \le \tau \l \grad \w{\mc{L}}(v) ,s \r\,.
\end{equation*}
By the above estimate and the line search condition \eqref{eq:new_line_search}, we conclude our claim.  

% and the convergence of $v_k$, we can conclude that for some $k_0$, $v_k \in V$ and $\alpha_k = 1$ for any $k \ge k_0$.  

\emph{Step 3.} We finally consider the local convergence rate with stepsize one. For this, we claim that there exists a neighborhood $V_0 \subset V$ of $v_*$ such that $R_v(s) \in V_0$ holds for any $v \in V_0$, where $s = - H(v)^{-1} \grad \w{\mc{L}}(v)$. Let 
\begin{equation*}
    r(v) := \grad \w{\mc{L}} (v) - P_{v_*, v} \hess \w{\mc{L}}(v_*)[R_{v_*}^{-1}v]\,.
\end{equation*}
By a direct computation with the relation $R_v^{-1}(v_*) = - P_{v_*,v}R_{v_*}^{-1}(v)$ \cite{boumal2022intromanifolds}, we have 
 \begin{align*}
  H(v)^{-1} \grad \w{\mc{L}} (v) + R_v^{-1}(v_*) = H(v)^{-1} \left(r(v) +  P_{v_*, v} \hess \w{\mc{L}}(v_*)[R_{v_*}^{-1}v] - H(v) P_{v_*,v}R_{v_*}^{-1}(v) \right),
\end{align*}
which gives the estimate 
\begin{align} \label{auxeq_1}
 \norm{s - R_v^{-1}(v_*)} & \le \norm{H(v)^{-1}} \left( \norm{r(v)} + \norm{P_{v_*, v} \hess \w{\mc{L}}(v_*) - H(v) P_{v_*,v}} \norm{R_{v_*}^{-1}(v)} \right).
\end{align}
By the Lipschitz property of $H(v)$ and $\norm{v - v_*} \le d(v,v_*)$, we estimate, for $v$ near $v_*$,
\begin{align} \label{auxeq_2}
     \norm{P_{v_*, v} \hess \w{\mc{L}}(v_*) - H(v) P_{v_*,v}} \le   \norm{P_{v_*, v} \hess \w{\mc{L}}(v_*) - H(v_*) P_{v_*,v}} + O(d(v,v_*))\,. 
\end{align}
Without loss of generality, we let $V_0$ be small enough such that \eqref{eq:localest} holds with constant $K$ defined in \eqref{eq:localconst}. Then it is easy to see from \eqref{auxeq_1} and \eqref{auxeq_2} that for some constant $C$,
\begin{align}  \label{est:rate_1}
    \limsup_{v \to v_*} \frac{d(R_v(s),v_*)}{d(v,v_*)} \le  \limsup_{v \to v_*} \frac{K \norm{s - R_v^{-1}(v_*)}}{a_0 \norm{R_{v_*}^{-1}(v)}} \le C \norm{\hess \mc{L}(v_*)  - H(v_*)} = O(\norm{[A,B]}^{1/2})\,,
\end{align}
 where we used $r(v) =  O\left(\norm{R_{v_*}^{-1}(v)}^2\right)$ by \eqref{eq:gradient_exp} and $\grad \w{\mc{L}}(v_*) = 0$, and the last estimate is from \eqref{auxeqlin} and \eqref{eq:approx_1}. When $\norm{[A,B]}$ is small enough, we obtain 
\begin{align} \label{est:rate_11}
 \limsup_{v \to v_*}  \frac{d(R_v(s),v_*)}{d(v,v_*)} < 1\,.  
\end{align}
In particular, if $[A,B] = 0$, it holds that $\hess \w{\mc{L}}(v_*) = H(v_*)$ by Lemma \ref{prop:ob2}. In this case, similarly, we have 
\begin{align} \label{est:rate_2}
      \limsup_{v \to v_*} \frac{d(R_v(s),v_*)}{d(v,v_*)^2} < + \infty\,.
\end{align}
We now have all the ingredients to finish the proof. Since $v_*$ is the accumulation point, there exists $k_0$ such that $v_{k_0} \in V_0$. Then by \emph{Step 2} and \emph{Step 3}, a simple induction argument yields $v_k \in V_0$ and $\alpha_k = 1$ for $k \ge k_0$. In view of  estimates \eqref{est:rate_1} and \eqref{est:rate_2}, we readily have the global convergence of $v_k$ and its convergence rate. 
\end{proof}

\begin{remark}
\mb{Empirically, the projected approximate Newton method in Algorithm \ref{alg:for ineq} exhibits a similar convergence rate when applied to \eqref{eq:problem with inequality constraint} with inequality constraints as it does for \eqref{eq:sub-problem}; see Section \ref{subsec:compare H} for the numerical results.}
Nevertheless, to the best of the author's knowledge, there exists very little research on the convergence of projection-type Riemannian optimization algorithms, and we choose to postpone the convergence analysis of Algorithm \ref{alg:for ineq} for \eqref{eq:problem with inequality constraint} to future investigations. We also refer interested readers to \cite{bergmann2019intrinsic,liu2020simple,yang2014optimality} for recent advancements in Riemannian optimization with constraints. 
\end{remark}

\subsection{Numerical stability} \label{subsec:stable}
We next show that for commuting matrices $(A, B) \in {\rm F}_{\rm com}$, our VJD algorithm can produce a reliable simultaneous approximate diagonalization. Note that when $[A,B] = 0$, for suitable $\ep$ in \eqref{eq:const}, each sub-optimization problem \eqref{eq:sub-problem} or \eqref{eq:problem with inequality constraint} in Algorithm \ref{alg:VJD} admits a minimizer $(\lad^{(j)},\mu^{(j)},v^{(j)})$ satisfying $\mc{L}(\lad^{(j)},\mu^{(j)},v^{(j)}) = 0$. However, due to the presence of noises and iteration errors, the input matrices may not exactly commute and $(\lad^{(j)},\mu^{(j)},v^{(j)})$ can not be perfectly solved. In this scenario, we assume that the sub-optimization problems \eqref{eq:sub-problem} and \eqref{eq:problem with inequality constraint} are solved by Algorithm \ref{alg:for ineq} with outputs $(\w{\lad}_j,\w{\mu}_j,\w{v}_j)$ satisfying
\begin{align} \label{eq:ini_assp}
     \mc{L}(\w{\lad}_j,\w{\mu}_j,\w{v}_j) = O(\d)\,,
\end{align}
for some precision parameter $\d > 0$, and prove that our Algorithm \ref{alg:VJD} can output matrices $U \in \on$ and $D_1,D_2 \in \diag{n}$ with a controlled error; see Proposition \ref{prop:numeri_sta}. \mb{We remark that the justification of the assumption \eqref{eq:ini_assp} relies on the convergence of Algorithm \ref{alg:for ineq} for \eqref{eq:problem with inequality constraint} and hence is left open.} 
The proof is based on the following two lemmas: a standard perturbation result for eigenvalue problems \cite[p.58]{rellich1969perturbation} and a backward error stability result for the common eigenvector problem \cite[Theorem 4]{Ghazi2008}. 

\begin{lemma}\label{lemma:perturbation}
Suppose that $A \in \sn$ has $n$ eigenvalues $\lambda_1 \le \cdots \le \lambda_n$ with the associated orthonormal eigenvectors $v_1,\ldots,v_n$. For any $\De A \in \sn$ with $\norm{\De A}$ small enough, let $\w{\lambda}_1 \le \cdots \le \w{\lambda}_n$ be the eigenvalues of $A + \De A$. Then, the perturbed matrix $A + \De A$ has orthonormal eigenvectors $\w{v}_i$ associated with $\w{\lad}_i$, which satisfy
\begin{equation}\notag
 |\w{\lad}_i - \lad_i| = O(\norm{\De A})\,,\q \norm{\w{v}_i - v_i} = O(\norm{\De A})\,.
\end{equation}
\end{lemma}

\begin{lemma} \label{lem:backstab}
Suppose that $v \in \S^{n-1}$ is a common eigenvector of matrices $A, B \in \sn$ with $A v = \lad v$ and $B v = \mu v$. If $\w{v}$ 
is an approximation of $v$, then there exist $\w{\lad},\w{\mu} \in \R$ and $\De A, \De B \in \sn$ such that 
\begin{align*}
 (A + \De A)\w{v} = \w{\lad} \w{v}\,,\q (B + \De B)\w{v} = \w{\mu} \w{v}\,,
\end{align*}
and
\begin{align*}
\sqrt{\frac{\norm{\De A}^2}{\norm{A}^2} + \frac{\norm{\De B}^2}{\norm{B}^2}} = \sqrt{
\frac{\norm{ A \w{v} - \l \w{v}, A \w{v}\r \w{v}}^2}{\norm{A}^2} + \frac{\norm{ B \w{v} - \l \w{v}, B \w{v}\r \w{v}}^2}{\norm{B}^2}
}\,.
\end{align*}
\end{lemma}

We now state and prove the main result of this section. 

\begin{proposition} \label{prop:numeri_sta}
For commuting matrices $A, B \in \sn$, suppose that the minimizers $(\lad^{(j)},\mu^{(j)},v^{(j)})$ in Algorithm \ref{alg:VJD} are approximated by $(\w{\lad}_j,\w{\mu}_j,\w{v}_j)$ computed from Algorithm \ref{alg:for ineq} with the estimate \eqref{eq:ini_assp}. Then, Algorithm \ref{alg:VJD} generates matrices $U \in \on$ and $D_1,D_2 \in \diag{n}$ that satisfy
\begin{align*}
    \norm{A - U D_1 U^T}^2 + \norm{B - U D_2 U^T}^2 = O(n \d) + O(n^2 \epsilon \sqrt{\d})\,,
\end{align*}
where $\epsilon$ is the threshold parameter in the stopping criterion of Algorithm \ref{alg:for ineq}.
\end{proposition}

\begin{proof}
We first estimate the error for the nearest
orthogonal matrix problem \eqref{eq:orthogonalize} involved in Algorithm \ref{alg:VJD}. From Algorithm \ref{alg:for ineq}, it is clear that $(\w{\lad}_j,\w{\mu}_j,\w{v}_j) = (\l \w{v}_j, A \w{v}_j  \r, \l \w{v}_j, B \w{v}_j\r, \w{v}_j)$, and then the assumption \eqref{eq:ini_assp} implies  
\begin{equation}\label{eq:stop criterion}
\w{\mathcal{L}}\bracket{\tilde{v}_j} = O(\d)\q \text{for all}\ j\,. 
\end{equation}
By viewing $\w{v}_j$ as an approximation to the common eigenvector $v^{(j)}$ and applying Lemma \ref{prop:numeri_sta}, there exist small perturbations $\De A_j, \De B_j \in \sn$ such that $\tilde{v}_j$ is an eigenvector of both $A+\Delta A_j$ and $B+\Delta B_j$ and there holds 
\begin{equation} \label{eq:perturb_order}
   \sqrt{\norm{\Delta A_j}^2 + \norm{\Delta B_j}^2} \le \frac{\max\{\norm{A},\norm{B}\}}{\min\{\norm{A},\norm{B}\}} \sqrt{\w{\mc{L}}(\w{v}_j)} = O(\sqrt{\d})\,.
\end{equation}
Therefore, by \eqref{eq:perturb_order} and Lemma \ref{lemma:perturbation}, we obtain 
\begin{align*}
    \norm{\w{v}_j - v_j}  = O(\sqrt{\d})\,,
\end{align*}
and hence 
\begin{equation}\label{eq:ineq}
    \abs{\l\tilde{v}_i,\tilde{v}_j\r} = O(\sqrt{\d})\q \text{for}\ i \ne j\,.
\end{equation}
Letting $V = [\w{v}_1,\ldots,\w{v}_n]$, it readily follows from \eqref{eq:ineq} that $(V^T V)_{ii} = 1$ and $ (V^T V)_{ij} = O(\sqrt{\d})$, which implies 
\begin{align} \label{eq:pert_svd}
    \norm{V^T V - I} \le \norm{V^T V - I}_{{\rm F}} = O(n\sqrt{\d})\,.
\end{align}
Again by Lemma \ref{lemma:perturbation}, the estimate \eqref{eq:pert_svd} above gives that all the eigenvalues of $V^T V$ are $O(n\sqrt{\d})$ perturbations of $1$, that is, all the diagonal entries of $\Sigma \in {\rm diag}(n)$ in the SVD $V = U_1 \Sigma U_2$ of $V$ are of the form $1 + O(n \sqrt{\d})$. This allows us to conclude 
\begin{align*}
 \norm{U_1 U_2 - V} = \norm{I - \Sigma} = O(n \sqrt{\d})\,.
\end{align*}

We now write $U = U_1 U_2 = [\Bar{v}_1,\cdots,\Bar{v}_n]$ and find
\begin{equation} \label{auxeqq}
    \sum_{j=1}^n \norm{\tilde{v}_j - \Bar{v}_j}^2  = \fnorm{U - V}^2 = \fnorm{I - \Sigma}^2 = O(n^3 \d)\,.
\end{equation}
By the stopping criterion in Algorithm \ref{alg:for ineq}, we directly see $\norm{\grad \mc{L}(\w{v}_j)} \le \epsilon$. Then, from \eqref{auxeqq}, we have
\begin{equation}\notag
    \begin{split}
        & \Big|\sum_{j=1}^n \w{\mathcal{L}}\bracket{\Bar{v}_j} - \sum_{j=1}^n \w{\mathcal{L}}\bracket{\tilde{v}_j} \Big| \le \sum_{j=1}^n  \Big|\w{\mc{L}}\bracket{\Bar{v}_j} - \w{\mc{L}}\bracket{\tilde{v}_j} \Big| \le O(\epsilon)\sum_{j=1}^n  \norm{\bar{v}_j - \w{v}_j} = O(n^2 \epsilon \sqrt{\d})\,,
    \end{split}
\end{equation}
which further gives 
\begin{align*}
 \sum_{j=1}^n \w{\mathcal{L}}(\Bar{v}_j) \le \sum_{j=1}^n \w{\mathcal{L}}(\tilde{v}_j) + O(n^2 \epsilon \sqrt{\d}) = O(n \d) + O(n^2 \epsilon \sqrt{\d})\,.
\end{align*}
The proof is completed by the following simple estimate:
\begin{align*}
     \norm{A - U D_1 U^T}^2 + \norm{B - U D_2 U^T}^2 \le \norm{A - U D_1 U^T}_{{\rm F}}^2 + \norm{B - U D_2 U^T}_{{\rm F}}^2 \le  \sum_{j=1}^n \w{\mathcal{L}}(\Bar{v}_j)\,,
\end{align*}
where the last inequality follows from the definitions of $U \in \on$ and $D_1,D_2 \in \diag{n}$. 
\end{proof}

\section{Numerical experiments}\label{sec:experiments}
In this section, we present extensive numerical experiments to verify the efficiency and robustness of our proposed algorithms in Section \ref{sec:algorithm}. In Section \ref{subsec:compare H}, we test the performance of the approximate Newton method (cf.\,Algorithm \ref{alg:for ineq}) with the exact Riemannian Hessian and various approximate Hessian matrices. Then in Section \ref{subsec:compare}, we compare the Jacobi algorithm and the VJD one for almost commuting matrices in detail, and we also investigate, both numerically and theoretically, the relations between the error cost function and the magnitude of the commutator. In Section \ref{subsec:ICA}, we apply 
our algorithm to the independent component analysis (ICA) problem. \mb{All the experiments presented below are conducted using Python on a personal laptop with 16 GB RAM and 8-core 2.6 MHz CPU. The implementation of the JADE algorithm is based on the Python package \cite{barachant2015pyriemann}. The synthetic almost commuting matrices in Sections \ref{subsec:compare H} and \ref{subsec:compare} are generated by $(A,B) = (A_*, B_*) + (\Delta A, \Delta B)$, where the commuting pair $(A_*, B_*) \in {\rm F}_{\rm com}$ is sampled from the distribution $\mathbb{P}_{{\rm F}_{\rm com}}$ introduced in Section \ref{sec:stability analysis} and $(\Delta A, \Delta B) \in \sn \t \sn$ is the additive Gaussian noise with noise level $\si$, namely, $(\Delta A)_{ij}$ and $(\Delta B)_{ij}$ with $i \le j$ are i.i.d. Gaussian $\mc{N}(0,\si^2)$.}

\subsection{Comparison between different approximate Hessian}\label{subsec:compare H}

\mb{We explore different approximate Hessian matrices within Algorithm \ref{alg:for ineq} and compare their performances for solving \eqref{eq:sub-problem} and \eqref{eq:problem with inequality constraint} in terms of the computing times.} As discussed after Algorithm \ref{alg:for ineq}, these approximate Hessian matrices can be broadly categorized into two families: those with the tangent space $T_v \S^{n-1}$ as their invariant subspaces and those without. For the first family, the search direction $s_k$ in \eqref{eq:quasi Newton equation} can be readily solved by the pseudoinverse: 
\begin{align} \label{eq:sch_1}
    s_k = - H(v_k)^\dag \grad \w{\mc{L}}(v_k)\,,
\end{align}
while for the second family, we consider the associated least-squares problem:
\begin{align} \label{eq:lssquare}
    \min_{s \in T_{v_k} \S^{n-1}} \norm{H(v_k) s + \grad \w{\mc{L}}(v_k)} = \min_{x \in \R^n} \norm{H(v_k)(I - v_k v_k^T)x + \grad \w{\mc{L}}(v_k)}\,.
\end{align}
It is well-known \cite{golub2013matrix} that the above problem \eqref{eq:lssquare} admits a minimizer:
\begin{align} \label{eq:sch_2}
s_k = - \left(I - v_k v_k^T\right)\left(H(v_k)(I - v_k v_k^T)\right)^\dag \grad \w{\mc{L}}(v_k)\,.
\end{align}
\mb{Here we consider the following four approximate Hessian matrices (which have been justified in Lemma \ref{prop:ob2} and the discussion afterward):  
the Euclidean Hessian of $\mc{L}(\Lad,v)$ in $v$: $H_0(v) = \nabla_v^2{\mathcal{L}}(\Lad, v)|_{(\Lad,v) = (\l v, A v \r,\l v ,B v\r,v)}$, the Riemannian Hessian of $\mc{L}$ in $v$ \eqref{eq:rieman_Lv}: $H_1(v) = \hess_v {\mathcal{L}}(\Lad,v)|_{(\Lad,v) = (\l v, A v \r,\l v ,B v\r,v)}$, the projected Euclidean Hessian of $\mc{L}$: $H_2(v) = (I_n - v v^T)H_0(v)(I_n - v v^T)$, the exact Riemannian Hessian of $\w{\mc{L}}(v)$: $H_3(v) = \hess {\w{\mathcal{L}}}(v)$.}

\mb{We test Algorithm \ref{alg:for ineq} with the above four candidates of the approximate Hessian $H(v)$ for solving \eqref{eq:sub-problem} and \eqref{eq:problem with inequality constraint} with $j = 2$ and let $s_k$ be updated accordingly either by \eqref{eq:sch_1} or \eqref{eq:sch_2}, for three 
randomly generated pairs of almost commuting matrices of dimension $n = 1000$ with $\norm{[A,B]}$ of orders $O(10^{-6})$, $O(10^{-4})$, and $O(10^{-2})$. The numerical results are represented in Figure \ref{fig:H}, which clearly shows that all the candidates achieve similar levels of accuracy with nearly the same iteration steps. However, 
compared to the alternating minimization and the Riemannian Newton method (i.e., Algorithm \ref{alg:for ineq} with approximate Hessian $H_1$ and $H_3$, respectively), Algorithm \ref{alg:for ineq} with $H_0$ or $H_2$ significantly reduces the computational cost, with the running time being about five times less than that of the standard Riemannian Newton method.} Note that the convergence for the case of $H_2$ \mb{has been established} in Theorem \ref{thm:glob_conv}. The algorithm with the approximate Hessian $H_0$, which appears to be the most efficient option, relies on the heuristic scheme \eqref{eq:sch_2} and does not align with the Riemannian optimization framework. Some additional techniques may be necessary to ensure its convergence properties. Figure \ref{fig:H} also verifies the robustness of the superiority of the approximate Hessians $H_0$ and $H_2$ with respect to the magnitude of $[A,B]$.

\begin{figure}[!htbp]
	\centering
		\begin{minipage}[t]{1.0\linewidth}
			\centering
			\includegraphics[width=1.0\linewidth]{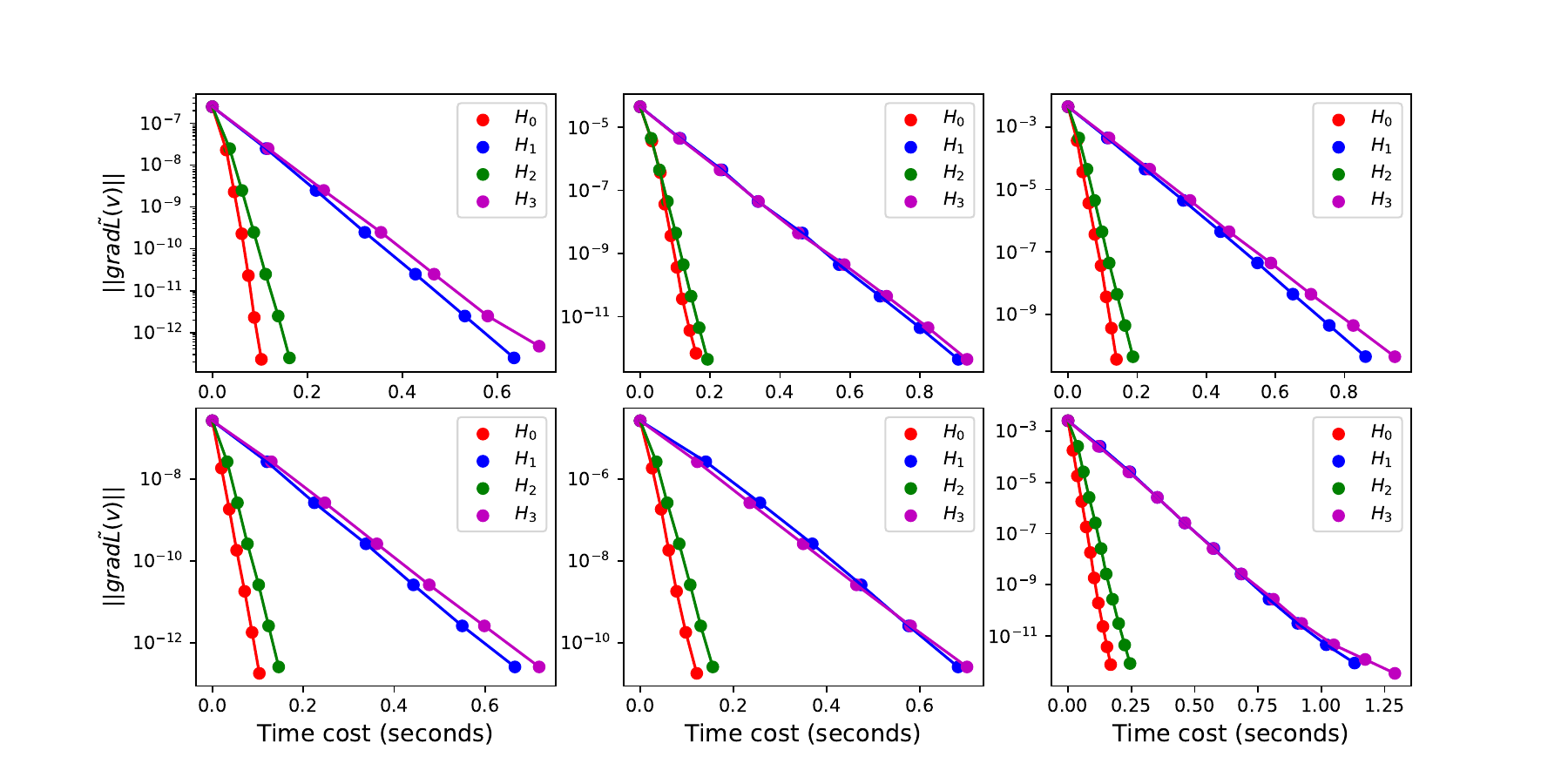}\\
			\vspace{0.02cm}
		\end{minipage}
	\caption{\mb{Convergence histories of Algorithm \ref{alg:for ineq} for \eqref{eq:sub-problem} (i.e., the first triplet) and \eqref{eq:problem with inequality constraint} with $j = 2$ (i.e., the second triplet)  with various approximate Hessian matrices $H$ for randomly generated almost commuting matrices $(A,B)$ of dimension $n= 1000$ with $\norm{[A,B]} = O(10^{-6}), O(10^{-4}), O(10^{-2})$ (left to right). The results for \eqref{eq:sub-problem} and \eqref{eq:problem with inequality constraint} with $j = 2$ are given in the first and second rows, respectively.}}
	\label{fig:H}
	\vspace{-0.2cm}
\end{figure}

\subsection{Comparison with Jacobi algorithm and relation with Lin's theorem}\label{subsec:compare}

\mb{We test Jacobi Algorithm \ref{alg:jacobi} and our VJD Algorithm \ref{alg:VJD} with approximate Hessian $H_0$ given in Section \ref{subsec:compare H} on randomly 
generated almost commuting symmetric matrices of dimensions $n = 50, 100, 500, 1000$ with noise levels $\si = 10^{-2}, 10^{-3},10^{-4},10^{-5}$, respectively, 
to compare their computational efficiencies.} We plot the distributions of the computing times for both algorithms across all the generated samples in Figure \ref{fig:time compare}. It clearly demonstrates that our VJD algorithm is robust against noise and much more efficient than the Jacobi one in the almost commuting regime, especially when the matrix size is large (say, $n=500,1000$). \mb{In addition, noting that Jacobi Algorithm \ref{alg:jacobi} involves $n(n-1)/2$ Givens rotation sweeps and in each sweep, $O(n)$ operations are needed 
to update matrices $A$ and $B$, thus the computational complexity of the Jacobi algorithm is $O(n^3)$. We numerically compare the complexities of the Jacobi algorithm and the proposed VJD one in Figure \ref{fig:time order}, where we plot the computational times of Algorithms \ref{alg:jacobi} and \ref{alg:VJD} as functions of the matrix size on a double logarithmic scale with base 10. It shows that in practice, our VJD algorithm is also of complexity $O(n^3)$ but with a smaller prefactor.}

\begin{figure}[!htbp]
	\centering
		\begin{minipage}[t]{0.7\linewidth}
			\centering
			\includegraphics[width=1.0\linewidth]{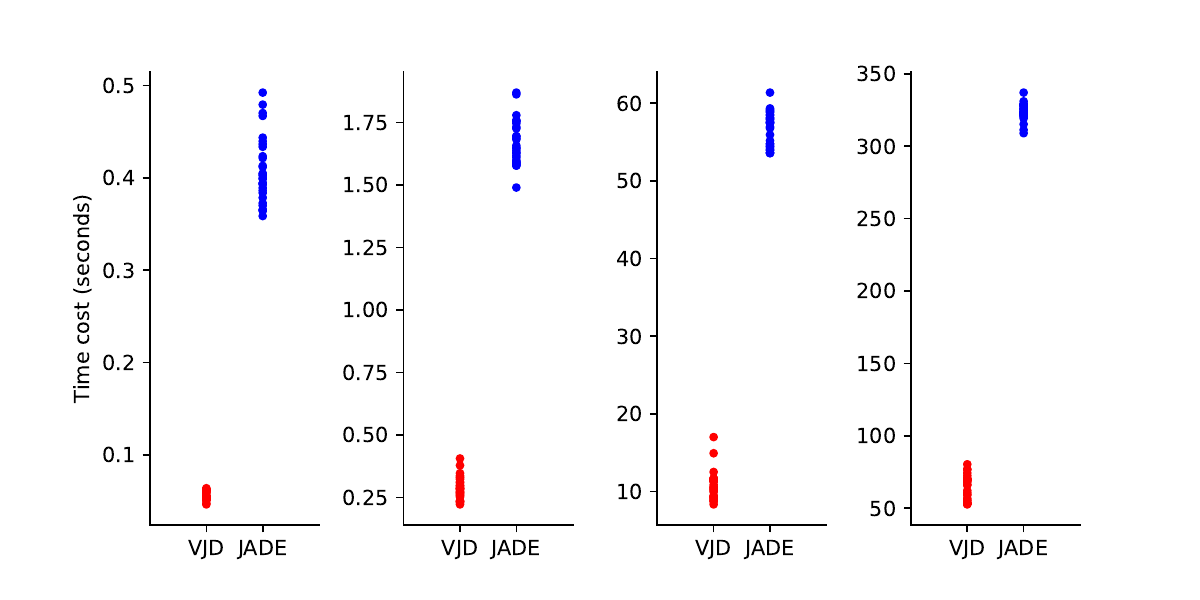}\\
			\vspace{0.02cm}
		\end{minipage}
	\caption{\mb{Statistical distributions of the computational times of Jacobi Algorithm \ref{alg:jacobi} (blue dots) and VJD Algorithm  \ref{alg:VJD}  (red dots) for $30$ randomly 
generated pairs of almost commuting matrices of dimensions $n = 50, 100, 500, 1000$ with noise levels $\si = 10^{-2}, 10^{-3},10^{-4},10^{-5}$ (from left to right), respectively.}}
	\label{fig:time compare}
	\vspace{-0.2cm}
\end{figure}

\begin{figure}[!htbp]
	\centering
		\begin{minipage}[t]{0.7\linewidth}
			\centering
			\includegraphics[width=0.8\linewidth]{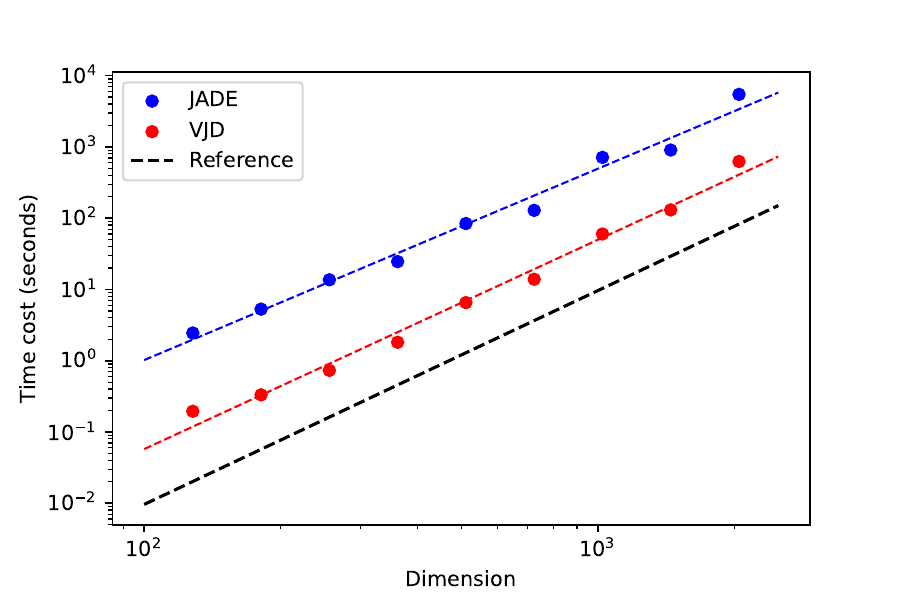}\\
			\vspace{0.02cm}
		\end{minipage}
	\caption{\mb{Averaged computational times of Jacobi Algorithm \ref{alg:jacobi} and VJD Algorithm \ref{alg:VJD} for almost commuting matrices with dimensions ranging from $128$ to $2048$ (log–log scale). For each dimension, $30$ samples are generated with noise level $\si = 0.01$. The blue and red dashed lines are obtained by the least-square fitting of the data points. The results are benchmarked by the black dashed line with a slope of $3$.}} 
	\label{fig:time order}
	\vspace{-0.2cm}
\end{figure}

Recalling that this work was initially motivated
by finding a numerically feasible solution to Lin's theorem, we are interested in whether our VJD algorithm can produce commuting matrices $(A',B')$ satisfying the estimate \eqref{eq:linest} in the almost commuting regime. \mb{For this, we consider the same pairs of almost commuting matrices $(A,B)$ as in Figure \ref{fig:time compare} and compute the numerical errors $\mc{J}(D_1,D_2, U)$ \eqref{eq:problem} of both Jacobi and VJD algorithms, denoted by $\mc{J}_{\rm Jacobi}$ and $\mc{J}_{\rm VJD}$, respectively. We plot $\mc{J}_{\rm VJD}$ as a function of the commutator $\norm{[A,B]}$ in Figure \ref{subfig:objective function} and find that $\mc{J}_{\rm VJD}$ scales as $\norm{[A,B]}^2$, which is independent of the problem size.} Thus, it is clear that when $\norm{[A,B]} \ll 1$, there holds
\begin{align} \label{eqscale}
    \mc{J}_{\rm VJD} \sim \norm{[A,B]}^2 \ll \norm{[A,B]}\,,
\end{align}
which means that $A' = U D_1 U^T$ and $B' = U D_2 U^T$ from Algorithm \ref{alg:VJD} would satisfy the bound \eqref{eq:linest}. We provide a theoretical lower bound of $\mc{J}_{\rm VJD}$ below to justify \eqref{eqscale}, which is generalized from \cite[Theorem 3.1]{glashoff2013matrix}. The proof is given in \ref{proof} for the sake of completeness. 

\begin{proposition}\label{thm:lower bound}
Let $A, B \in \sn$ be symmetric matrices with $\norm{A}, \norm{B} \le 1$. Suppose that matrices $U \in \on$ and $D_1, D_2 \in \diag{n}$ are computed by Algorithm \ref{alg:VJD} with input $(A,B)$. Then it holds that 
\begin{equation}\label{eq:lower bound}
\mathcal{J}(D_1,D_2,U) \ge \frac{1}{8}\norm{[A,B]}^2. 
\end{equation}
\end{proposition}

\noindent \mb{We also compute the relative errors between Jacobi and VJD algorithms in the same experimental setup by $|\mc{J}_{\rm Jacobi} - \mc{J}_{\rm VJD}|/\mc{J}_{\rm Jacobi}$ with the results shown in Figure \ref{subfig:relative error}. We can see that our VJD algorithm can achieve nearly the same accuracy as the Jacobi method but with much less computational time (cf. Figures \ref{fig:time compare} and \ref{fig:time order}).}

\begin{figure}[!htbp]
	\centering
	\subfigure[Error of VJD]{
		\begin{minipage}[t]{0.49\linewidth}
			\centering
			\includegraphics[width=1.0\linewidth]{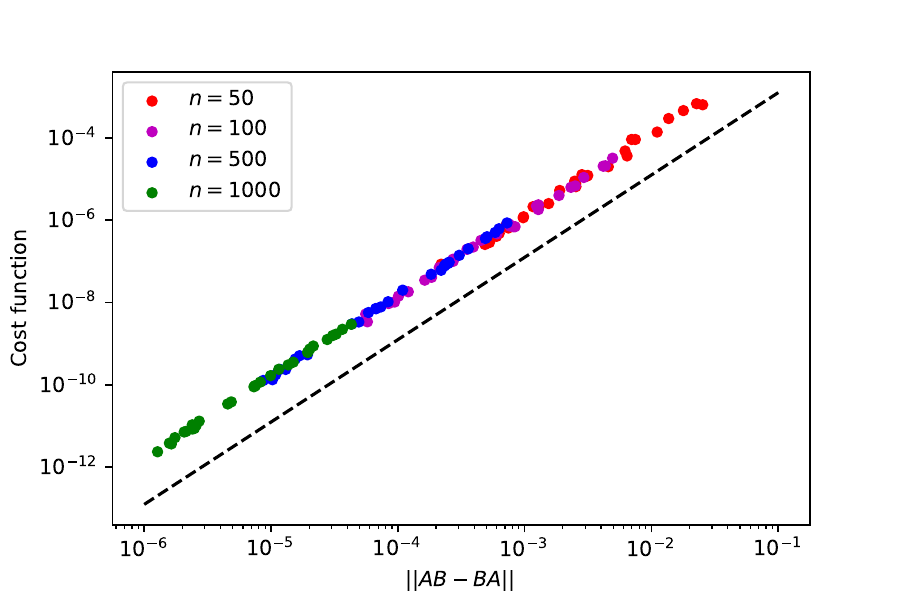}\\
			% \hspace{0.02cm}
			\label{subfig:objective function}
		\end{minipage}%
	}%
	\subfigure[Relative error]{
		\begin{minipage}[t]{0.49\linewidth}
			\centering
			\includegraphics[width=1.0\linewidth]{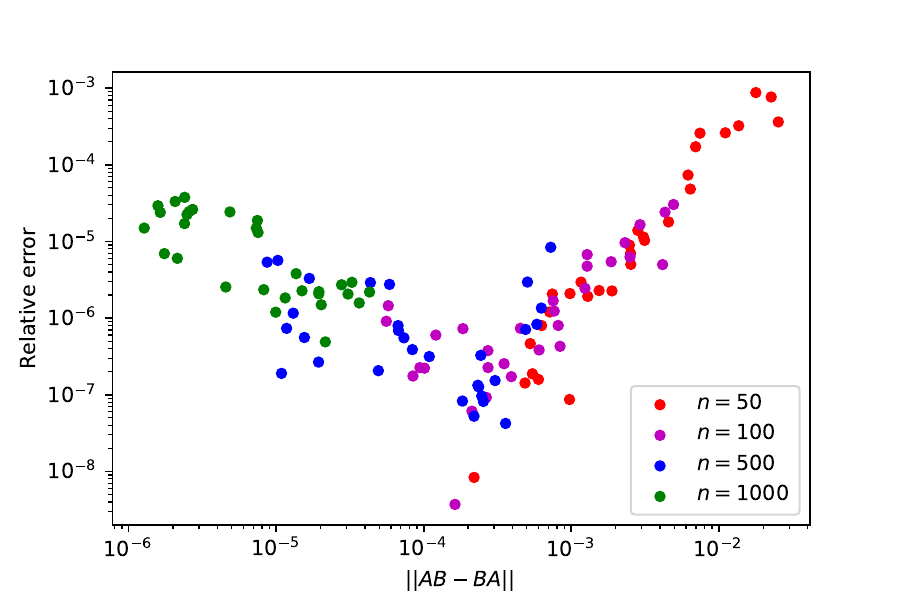}\\
			% \hspace{0.02cm}
			\label{subfig:relative error}
		\end{minipage}%
	}	
	\caption{\mb{(a) Errors $\mc{J}_{\rm VJD}$ of VJD Algorithm \ref{alg:VJD} for almost commuting matrices of various dimensions used in Figure \ref{fig:time compare} (log–log scale). The black dashed line represents the lower bound $\norm{[A,B]}^2/8$ in \eqref{eq:lower bound}. (b) Relative errors $|\mc{J}_{\rm Jacobi} - \mc{J}_{\rm VJD}|/\mc{J}_{\rm Jacobi}$ between Jacobi and VJD algorithms (log-log scale).}}
	\label{fig:error compare}
	\vspace{-0.2cm}
\end{figure}

\subsection{Application in independent component analysis}\label{subsec:ICA}

Let $s_1, \ldots, s_n \in \R^t$ be independent pure signals with length $t$, \mb{which are assumed to be non-Gaussian.} Suppose that the mixed signals $x_1, \ldots, x_r (r \ge n)$
are the linear combinations of the pure signals, i.e., $X = A S$, where $X$ and $S$ are matrices consisting of rows $x_i$ and $s_i$, respectively. 
The goal of ICA is to reconstruct the source signals $S$ and the mixing matrix $A$ \mb{from the noisy measured data $X$.} Cardoso et al. \cite{cardoso1998multidimensional,Cardoso1993} proposed a standard framework for solving this problem based on the joint diagonalization. Here we follow the setup in \cite{Rutledge2013} and apply the VJD algorithm to the ICA problem. The basic procedures are summarized as follows for the reader's convenience.  First, normalize $X$ such that each row of $X$ has zero mean, and define the whitened matrix $P_W = \sqrt{t}P$ with $P$ computed from the SVD of $X$: $X = Q \Sigma P^T$. 
Second, compute the fourth-order cumulant tensor of the matrix $P_W$ by 
  \begin{equation}
        K(i,j,k,l) = \l p_i\circ p_j\circ p_k\circ p_l \r - \l p_i \circ  p_j \r \l p_k \circ p_l\r - \l p_i \circ p_k \r \l p_j \circ p_l \r - \l p_i \circ p_l \r \l p_j \circ p_k \r\,,
    \end{equation}
where $p_i$ are the columns of $P_W$; 
notation $\circ$ represents the element-wise product (Hadamard product); and $\l \dd \r$ means the expected value (average). Third, project the cumulant tensor $K$ onto a set of $m=n(n+1)/2$ orthogonal eigenmatrices of dimension $n\times n$, where $n$ of them are zero matrices with one diagonal entry being $1$, and the others are zero matrices with two symmetrically off-diagonal entries being $\sqrt{0.5}$. Denote the projected matrices by $M_1,\cdots,M_m$. Fourth, apply the VJD Algorithm~\ref{alg:VJD} to 
jointly diagonalize the $m$ matrices $M_1,\cdots,M_m$. Denote the output orthogonal matrix by $U$. Finally, reconstruct the source signals and the mixing matrix $A$ by 
    \begin{equation}\notag
        \tilde{S} = \sqrt{t} U\Sigma^{-1}Q^TX\,,\q   \tilde{A} = X\tilde{S}^T\bracket{\tilde{S}\tilde{S}^T}^{-1}\,.
    \end{equation}

\mb{To test our VJD algorithm in the ICA application, we consider simulated one-dimensional and two-dimensional source signals in Figures \ref{subfig:source signal} and \ref{subfig:2d source signal}. We assume that the noisy measured signals are given by $X = A S + E$, where $A$ is a random orthogonal matrix and $E$ is a random matrix with $E_{ij}$ being i.i.d. Gaussian $\mc{N}(0,\eta^2)$; see Figures \ref{subfig:measured signal} and \ref{subfig:2d measured signal}. Then, in our experiments, there are $21$ symmetric matrices of size $6 \t 6$ to be jointly diagonalized. We plot in Figures \ref{fig:ICA} and \ref{fig:ICA2} the reconstructed signals by Jacobi Algorithm \ref{alg:jacobi} and VJD Algorithm \ref{alg:VJD} to compare their performance, which clearly shows that both of these two algorithms can help recover the source signals effectively. In addition, we repeat the experiments for the same source signals 
with 100 samples of $A$ and $E$ and record the average time cost and $L^2$-error $\norm{X_{\rm recons} - X_{\rm true}}_{\rm F}$ between the reconstructed signals $X_{\rm recons}$ and the source signals $X_{\rm true}$ in Tables \ref{tab:tableA} and \ref{tab:tableB}. It again demonstrates the advantages of our VJD algorithm over the standard Jacobi one.}

% To test the JADE algorithm with VJD for the jointly approximation diagonalizing the eigen-matrices step, we work on three statistically independent pure signals as source signals, and three mixed signals as measured signals, which are linear combinations of the three source signals with Gaussian noise. The results are shown in Figure~\ref{fig:ICA}. 

\begin{figure}[!htbp]
	\centering
	\subfigure[Source signals]{
		\begin{minipage}[t]{0.49\linewidth}
			\centering
			\includegraphics[width=1\linewidth]{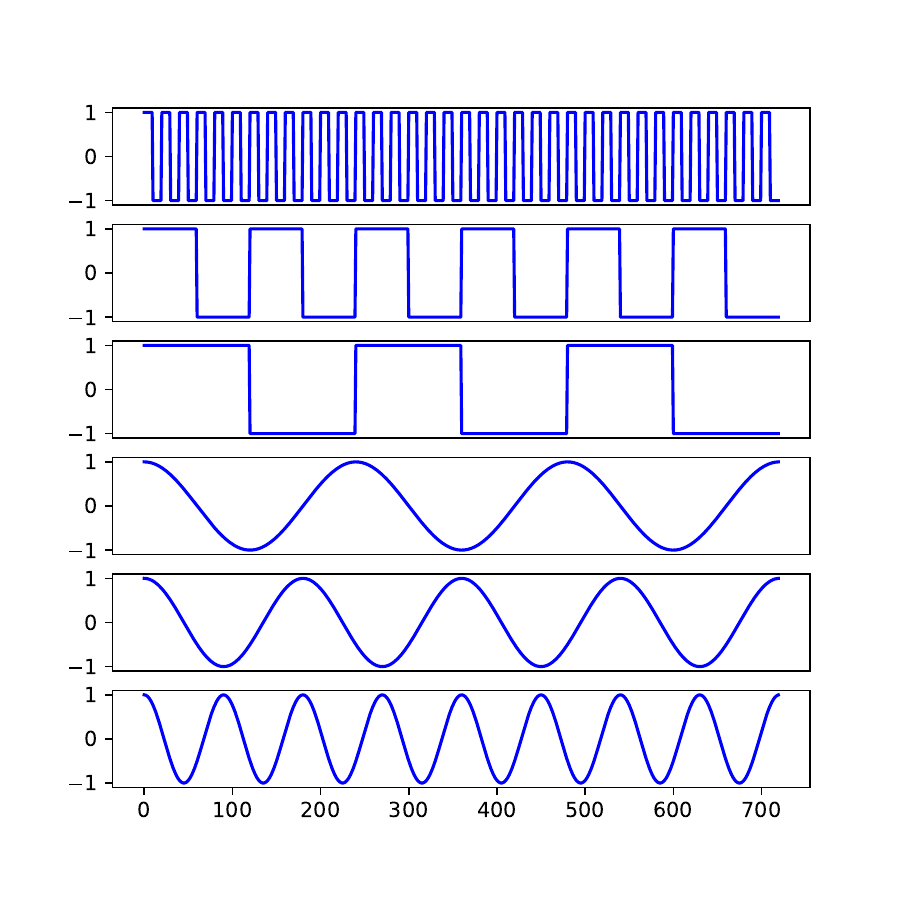}\\
			\label{subfig:source signal}
		\end{minipage}%
	}%
	\subfigure[\mb{Measured signals with noise level $\eta = 0.1$}]{
		\begin{minipage}[t]{0.49\linewidth}
			\centering
			\includegraphics[width=1\linewidth]{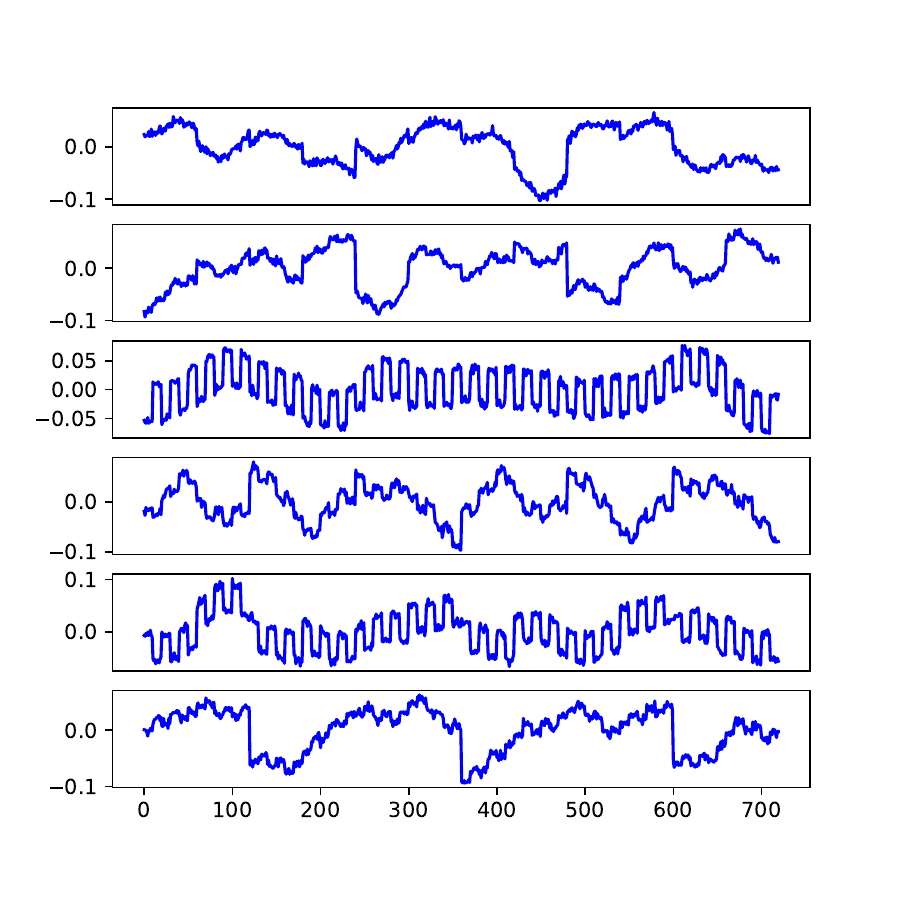}\\
			% \hspace{0.02cm}
			\label{subfig:measured signal}
		\end{minipage}%
	}%
        \vskip\parindent \vspace{-5 mm}
        \subfigure[Reconstruction by JADE]{
		\begin{minipage}[t]{0.49\linewidth}
			\centering
			\includegraphics[width=1\linewidth]{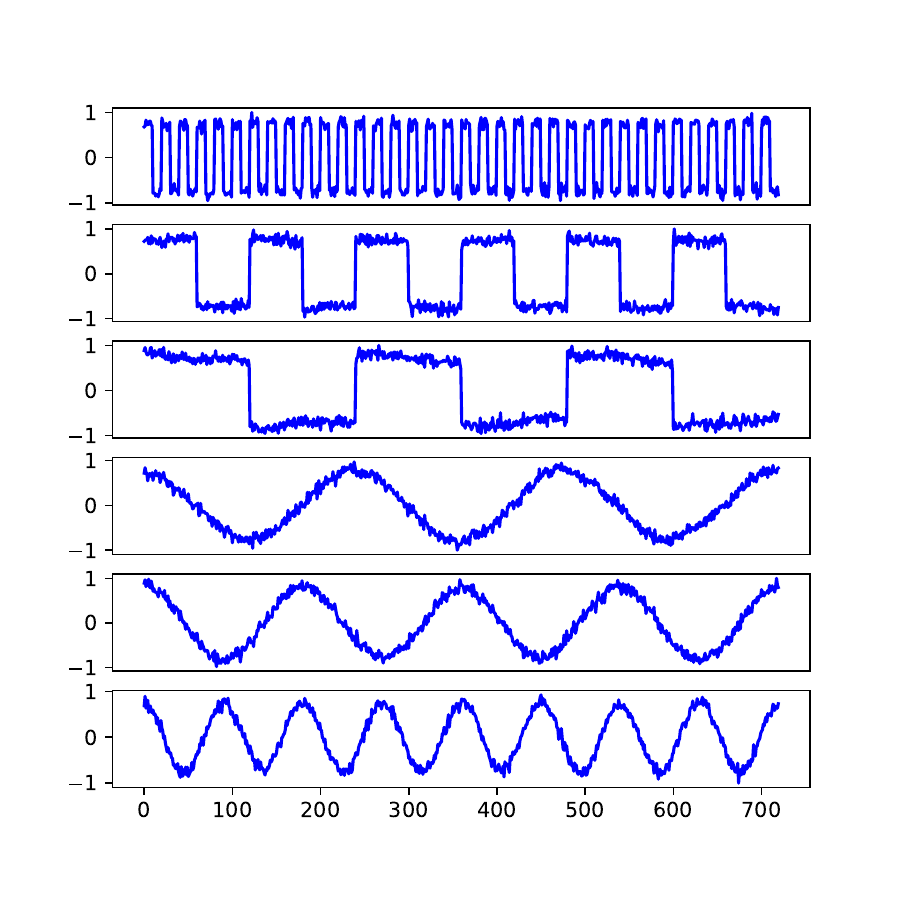}\\
			% \hspace{0.02cm}
			\label{subfig:reconstruction by JADE}
		\end{minipage}%
	}%
        \subfigure[Reconstruction by VJD]{
		\begin{minipage}[t]{0.49\linewidth}
			\centering
			\includegraphics[width=1\linewidth]{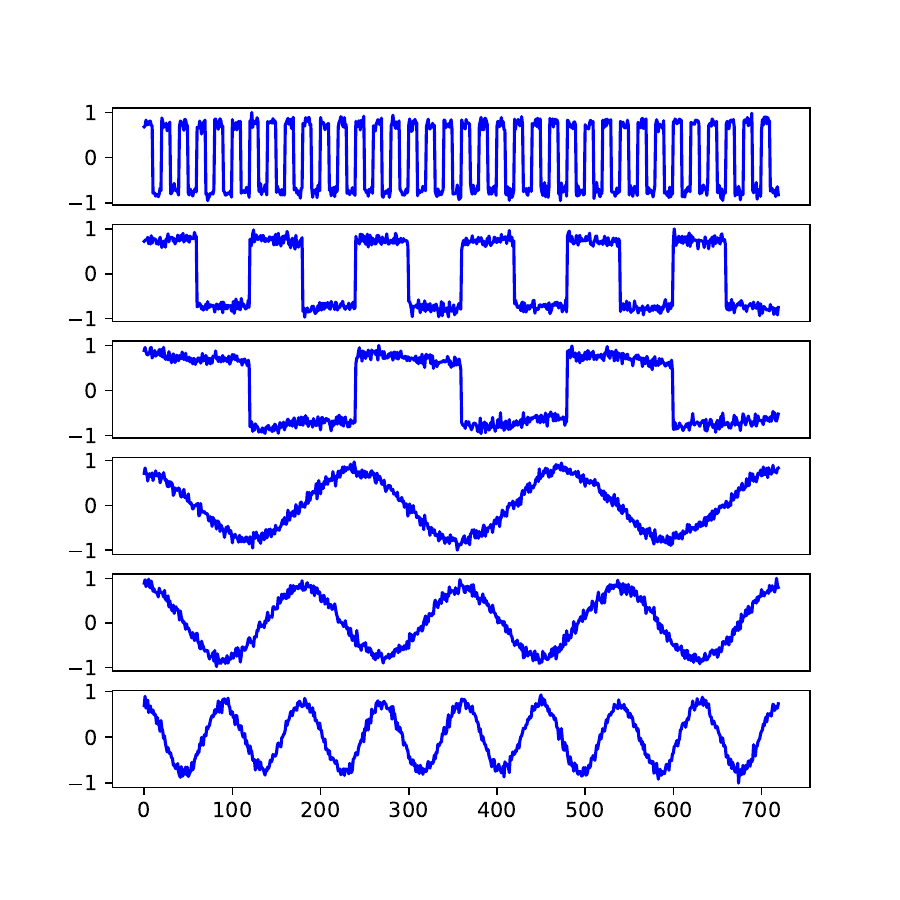}\\
			% \hspace{0.02cm}
			\label{subfig:reconstruction by VJD}
		\end{minipage}%
	}%
	
	\caption{\mb{ICA for one-dimensional signals. The operator norms of pairwise commutators $\norm{[M_i, M_j]}$ of matrices $\{M_j\}_{j = 1}^m$ to be jointly diagonalized range from $0.041$ to $0.319$.}}
	\label{fig:ICA}
	\vspace{-0.2cm}
\end{figure}

% As shown in Figure~\ref{subfig:reconstructed signal}, despite a little noise, we perfectly reconstruct the source signals. 

% We also implement the ICA algorithm on two-dimensional data as shown in Figure~\ref{fig:2d ICA}. Figures in the three lines represent source signals, measured signals and reconstructed signals, respectively. When implementing ICA algorithm on two-dimensional data, we should reshape each signal into one-dimensional, and restore all the data in a two-dimensional matrix. Then, the problem becomes the one-dimensional case as we have solved previously. 

\begin{figure}[!htbp]
	\centering
	\subfigure[Source signals]{
		\begin{minipage}[t]{0.49\linewidth}
			\centering
			\includegraphics[width=1\linewidth]{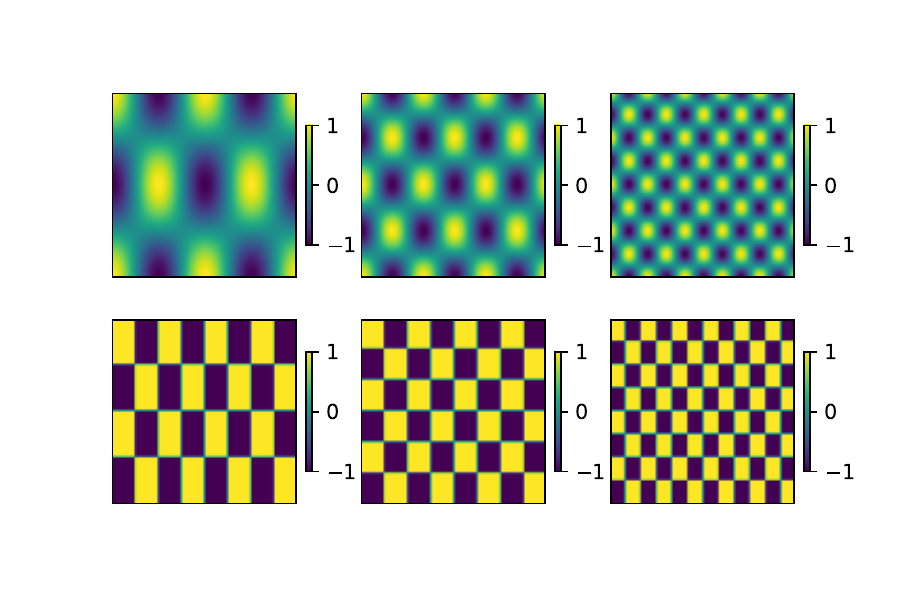}\\
			% \hspace{0.02cm}
			\label{subfig:2d source signal}
		\end{minipage}%
	}%
	\subfigure[\mb{Measured signals with noise level $\eta = 0.01$}]{
		\begin{minipage}[t]{0.49\linewidth}
			\centering
			\includegraphics[width=1\linewidth]{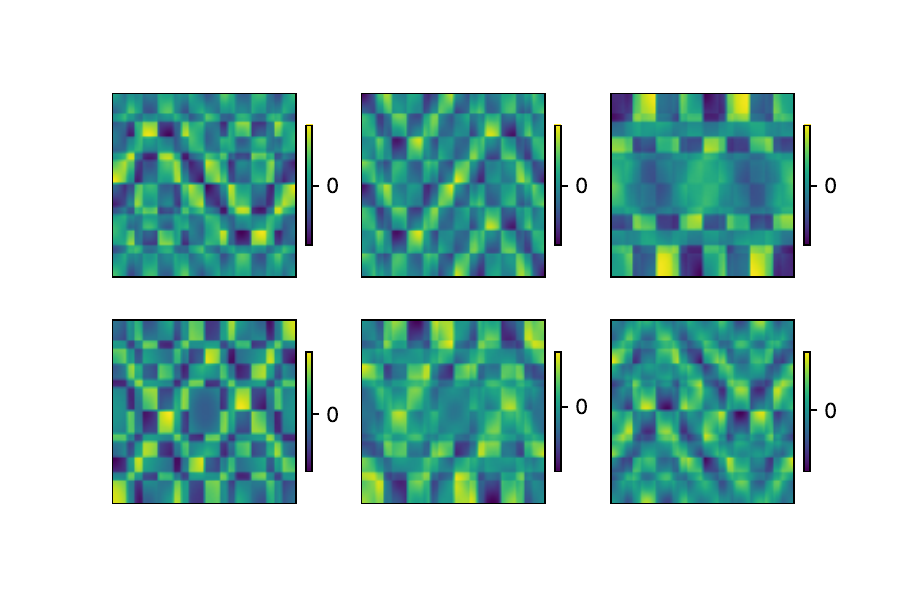}\\
			% \hspace{0.02cm}
			\label{subfig:2d measured signal}
		\end{minipage}%
	}%
   \vskip\parindent \vspace{-5 mm}
        \subfigure[Reconstruction by JADE]{
		\begin{minipage}[t]{0.49\linewidth}
			\centering
			\includegraphics[width=1\linewidth]{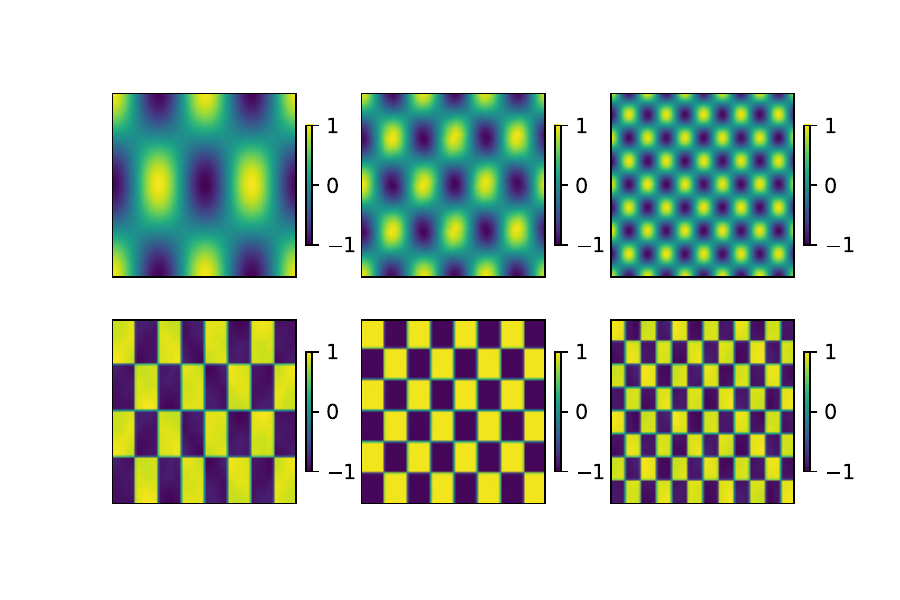}\\
			% \hspace{0.02cm}
			\label{subfig:2d reconstruction by JADE}
		\end{minipage}%
	}%
        \subfigure[Reconstruction by VJD]{
		\begin{minipage}[t]{0.49\linewidth}
			\centering
			\includegraphics[width=1\linewidth]{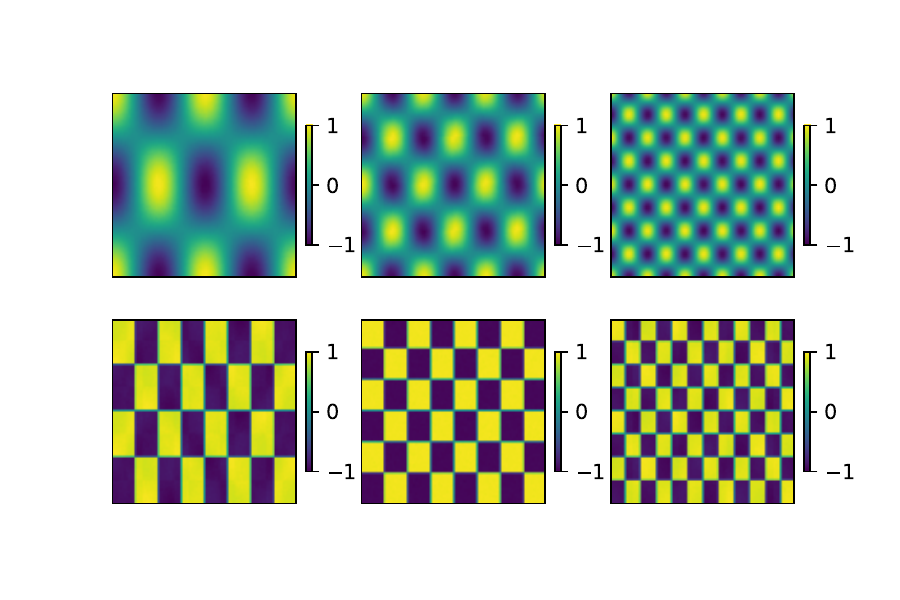}\\
			% \hspace{0.02cm}
			\label{subfig:2d reconstruction by VJD}
		\end{minipage}%
	}%
	
	\caption{\mb{ICA for two-dimensional signals. The operator norms of pairwise commutators $\norm{[M_i, M_j]}$ of matrices $\{M_j\}_{j = 1}^m$ to be jointly diagonalized range from $0.112$ to $0.669$.}}
	\label{fig:ICA2}
	\vspace{-0.2cm}
\end{figure}

\begin{table}[!htbp]
  \begin{minipage}{0.49\textwidth} % Adjust the width as needed
    \centering
    \begin{subtable}
      \centering
    \begin{tabular}{|c|c|c|} 
\hline
 & Avg. time (ms) & Avg. $L^2$-error \\ 
\hline
JADE & $5.526$ & $0.227$ \\ 
\hline
VJD & $2.143$ & $0.228$ \\ 
\hline
\end{tabular}
      \caption{\mb{Computing time and $L^2$-error for ICA in Figure \ref{fig:ICA}}}
      \label{tab:tableA}
    \end{subtable}
  \end{minipage}%
  \begin{minipage}{0.49\textwidth} % Adjust the width as needed
    \centering
    \begin{subtable}
      \centering
   \begin{tabular}{|c|c|c|} 
\hline
 & Avg. time (ms) & Avg. $L^2$-error \\ 
\hline
JADE & $6.935$ & $5.602 \times 10^{-2}$ \\ 
\hline
VJD & $2.493$ & $4.813 \times 10^{-2}$ \\ 
\hline
\end{tabular}
      \caption{\mb{Computing time and $L^2$-error for ICA in Figure \ref{fig:ICA2}}}
      \label{tab:tableB}
    \end{subtable}
  \end{minipage}
  \label{tab:bothtables}
\end{table}

\section{Conclusion}
The approximate joint diagonalization of almost commuting matrices arises in many applications and closely relates to the celebrated Huaxin Lin's theorem.  In this work, we have designed a vector-wise algorithm framework for jointly diagonalizing a given pair of almost commuting matrices $(A,B)$, which relies on solving the sub-optimization problems \eqref{eq:sub-problem} and \eqref{eq:problem with inequality constraint}. To motivate our VJD algorithm (cf.\,Algorithm \ref{alg:VJD}), we have first analyzed the stability of the local minimizers of $\mc{L}(\Lad,v)$ in \eqref{eq:sub-problem}, and then, with the help of 
the random matrix theory, we have shown that for almost commuting random matrices, with high probability the functional $\mc{L}$ has $n$ local minimizers with almost orthogonal minimizing vectors (cf.\,Theorem \ref{thm:generic_almost}). \mb{To efficiently solve the sub-optimization problems, we have proposed a projected 
Riemannian approximate Newton method (cf.\,Algorithm \ref{alg:for ineq}) for \eqref{eq:sub-problem} and \eqref{eq:problem with inequality constraint}.}  Moreover, we have proved in Theorem \ref{thm:glob_conv} that in the almost commuting regime, Algorithm \ref{alg:for ineq} for \eqref{eq:sub-problem} exhibits linear convergence with a rate of $O(| [A, B] |^{1/2})$, 
transitioning to quadratic convergence when $A$ and $B$ commute. \mb{However, the convergence of the projected Newton method (i.e., Algorithm \ref{alg:for ineq} for \eqref{eq:problem with inequality constraint}) is still open and needs further investigation.} We have also
proved that our VJD algorithm is stable with respect to iteration errors and noise. 

Numerical results suggest that our vector-wise framework is more efficient than the popular Jacobi-type algorithm, especially for large-scale matrices, and can be applied to the ICA problem for signal reconstruction. We have numerically and theoretically examined the relations between the cost functional $\mc{J}(D_1,D_2,U)$ in \eqref{eq:problem} and the commutator $[A,B]$ and find $\mc{J} \sim \norm{[A,B]}^2$, which indicates that when $\norm{[A,B]}$ is small enough, the commuting matrices $(A', B') = (U D_1 U^T, U D_2 U^T)$ constructed from Algorithm \ref{alg:VJD} satisfy the estimate \eqref{eq:linest} and hence give a desired numerical solution for Lin's theorem. Finally, we would like to point out that though our approach works quite well for almost commuting matrices, it would be still important and useful to find an algorithm for Lin’s Theorem beyond the almost commuting regime.

% but the failure probability of Algorithm \ref{alg:VJD} for random pairs of non-commutative matrices increases when $\norm{[A,B]}$ grows. 

% Besides, we provide a lower bound and an upper bound of the objective function, controlled by the commutator of the matrices and the precision of the optimization, respectively.

% For almost commuting real symmetric matrices, both algorithms provide orthogonal transformations with consistent objective function $\mathcal{J}(U)$, 

% while our algorithm costs much less time than the Jacobi algorithm does when the dimension is large, say $n=500,1000$. 

% Our algorithm can also be applied to commonly-used JADE algorithm for independent component analysis (ICA). \mr{[to revise]}

% In this paper, we study a vector-wise joint diagonalization (VJD) algorithm, where common or almost common eigenvectors of both matrices are found individually. The constructing of common eigenvectors can be restated as an optimization problem restricted on complex unit sphere, which is solved using projected Newton's method, which is proved locally convergent superlinearly. 

% The algorithm also works for special cases where the matrices to be jointly diagonalized are real symmetric, and can be generated to $m$-matrix cases. For those cases where the matrices are not strictly but almost commuting, we prove that the local minimzers of our objective function of the column-wise problem are stable under perturbation.

\section*{Declaration of competing interest}
The authors declare that they have no known competing financial interests or personal relationships that could have appeared to influence the work reported in this work.

\section*{Acknowledgement} 
The work of BL and JL is supported in part by the US National Science Foundation under award CCF-1910571 and by the US Department of Energy via grant DE-SC0019449. JL would like to thank Terry Loring for helpful discussions regarding the problem of almost commuting matrices.

\appendix
\section{Proof of Proposition \ref{thm:lower bound}} \label{proof}

\begin{proof}[Proof of Proposition \ref{thm:lower bound}]
We recall the construction of the diagonal matrices $D_1, D_2$ in Algorithm \ref{alg:VJD}:
\begin{align*}
    D_1 = {
    \rm diag}(\l v_1, A v_1\r, \ldots, \l v_n, A v_n\r)\,,\q  D_2 = {
    \rm diag}(\l v_1, B v_1\r, \ldots, \l v_n, B v_n\r)\,,
\end{align*}
where $v_i$ is the $i$-th column of the orthogonal matrix $U$. Then it is easy to see that
\begin{align*}
    J(D_1, D_2, U) = \norm{U^T A U - \diag{U^T A U }}^2 + \norm{U^T B U - \diag{U^T B U }}^2 = : \mc{J}(U)\,.
\end{align*}
Here and in what follows, $\diag{M}$ means the diagonal part of a $n \t n$ matrix $M$.
Hence, to bound $J(D_1, D_2, U)$ from below, it suffices to consider the lower bound of 
\begin{align} \label{auxeqop}
    \min \left\{\mc{J}(U)\,; \ U \in \on \right\}\,.
\end{align}
Let $V$ be a minimizer to the optimization problem \ref{auxeqop}, and we write 
\begin{equation}\label{decompose}
V^T A V = D_A + X\,,\q  V^T B V = D_B + Y\,.
% \begin{split}
% & V^*AV = D_A + X, \\
% & V^*BV = D_B + Y,
% \end{split}
\end{equation}
where $D_A: = {\rm diag}(V^T A V)$ and $D_B = {\rm diag}(V^T B V)$. Noting that $\norm{A} = \max_{v \in \S^{n-1}} |\l x, A x \r|$ holds for $A \in \sn$, we readily have, by definition,
\begin{equation} \label{eq:auxnorm}
\norm{D_A} \le \norm{A} \le 1\,,\q  \norm{D_B} \le \norm{B} \le 1\,.
% & \specnorm{D_A}\le\specnorm{V^TAV}=\specnorm{A}\le1, 
% & \specnorm{D_B}\le\specnorm{V^TBV}=\specnorm{B}\le1.
\end{equation}
The decomposition \eqref{decompose} can be rewritten as $A=VD_AV^T+VXV^T$ and $A=VD_BV^T+VYV^T$, which implies 
\begin{equation}\notag
\begin{split}
& AB = VD_AD_BV^T+VD_AYV^T+VXD_BV^T+VXYV^T, \\
& BA = VD_BD_AV^T+VD_BXV^T+VYD_AV^T+VYXV^T.
\end{split}
\end{equation}\notag
Then, a direct computation gives 
\begin{equation}
\begin{split}
[A,B]&=AB-BA=V\bracket{[D_A,D_B]+[D_A,Y]+[X,D_B]+[X,Y]}V^T \\
& =V\bracket{[D_A,Y]+[X,D_B]+[X,Y]}V^T \\
& =V\bracket{[D_A+X,Y]+[X,D_B]}V^T\,.
\end{split}
\end{equation}
By the triangle inequality and the unitary invariance of operator norm, there holds 
\begin{equation}\notag
\norm{[A,B]}\le\norm{[D_A+X,Y]}+\norm{[X,D_B]}\,,
\end{equation}
which further yields, by a simple estimate using \eqref{eq:auxnorm},
\begin{align*}
\norm{[A,B]}&\le 2\norm{D_A+X}\norm{Y} + 2\norm{X}\norm{D_B}  \le 2\bracket{\norm{X}+\norm{Y}} \le 2 \sqrt{2 \mc{J}(V)}\,.
\end{align*}
The proof is complete. 

% Therefore, 
% \begin{equation}\notag
% \begin{split}
% \norm{[D_A+X,Y]}&\le2\norm{D_A+X}\norm{Y}=2\norm{A}\norm{Y}\le2\norm{Y},\\
% \norm{[X,D_B]}&\le2\norm{X}\norm{D_B}\le2\norm{X}.
% \end{split}
% \end{equation}
% This implies

% which is equivalent to \eqref{eq:lower bound}. 
\end{proof}

\end{document}